\documentclass[11pt]{amsart}
\usepackage[foot]{amsaddr}
\usepackage[utf8]{inputenc}
\usepackage{amsmath}
\usepackage{amsfonts}
\usepackage{amssymb}
\usepackage{amsthm}
\usepackage{geometry}
\usepackage{enumitem}
\usepackage[dvipsnames]{xcolor}
\usepackage[utf8]{inputenc}
\usepackage[symbol]{footmisc}
\usepackage[centercolon]{mathtools}
\usepackage{microtype}
\usepackage{lmodern}
\usepackage[colorlinks=true, linkcolor=blue, citecolor=red]{hyperref}
\newcommand{\eps}{\varepsilon}

\newcommand{\grad}{\nabla}
\newcommand{\R}{\mathbb{R}}
\renewcommand{\div}{\grad\cdot}
\newtheorem{theorem}{Theorem}
\newtheorem{lemma}{Lemma}
\newtheorem{proposition}{Proposition}

\newtheorem{definition}{Definition}
\newtheorem{example}{Example}
\newtheorem{remark}{Remark}
\renewcommand{\L}{\mathcal{L}}
\newcommand{\intrd}{\int_{\mathbb{R}^d}}
\newcommand{\rd}{\mathbb{R}^d}
\newcommand{\id}{\textrm{id}}
\newcommand{\opt}{\textrm{opt}}

\title[Optimal stability estimates for advection-diffusion equations]{Optimal stability estimates and a new uniqueness result for advection-diffusion equations}
\author{Víctor Navarro-Fernández \and André Schlichting \and Christian Seis}
\email{\{victor.navarro,a.schlichting,seis\}@uni-muenster.de}
\address{Institut f\"ur Analysis und Numerik,  Westf\"alische Wilhelms-Universit\"at M\"unster \newline
Orl\'eans-Ring~10, 48149 M\"unster, Germany.}
\date{\today}

\begin{document}

\numberwithin{equation}{section}

\begin{abstract}
This paper contains two main contributions. First, it provides optimal stability estimates for advection-diffusion equations in a setting in which the velocity field is Sobolev regular in the spatial variable. This estimate is formulated with the help of Kantorovich--Rubinstein distances with logarithmic cost functions. Second, the stability estimates are extended to the advection-diffusion equations with velocity fields whose gradients are singular integrals of $L^1$ functions entailing a new well-posedness result.
\end{abstract}

\maketitle

%\tableofcontents

\section{Introduction}
\label{s:intro}

The advection-diffusion equation is of remarkable importance in many different physical contexts. For instance, one can consider a scalar parameter $\theta$ that represents the temperature of a fluid or a mass density. Then the advection-diffusion equation will give information about the evolution in time and space of this physical quantity when it is subjected to an advection field $u$ and when diffusion plays a role due to the molecular motion of the involved particles. Some specific examples of physical phenomena modeled with this equation include drift-diffusion processes in semiconductor physics,  heat transmission through a fluid layer, or mixing by stirring in industrial applications.

The Cauchy problem for the advection-diffusion equation is the following
\begin{equation}
	\left\lbrace
	\begin{array}{rcll}
		\partial_t\theta + \nabla\cdot (u\theta) & = & \kappa \Delta \theta &\quad \text{ in } (0,T)\times\rd,  \\
		\theta(0,\cdot) & = & \theta^0 &\quad \text{ in } \rd,
	\end{array}
	\right.
	\label{eq:adv-diff}
\end{equation}
driven by some vector field $u:(0,T)\times \R^d \to \R^d$, started at some given initial datum $\theta^0$ and with a given diffusion coefficient $\kappa>0$.

The mathematical theory in the setting of smooth vector fields is contained in the general classical theory for parabolic equations; see, for instance, Ladyzhenskaya et al.~\cite{Lady68}. Thanks to the linearity of the equations, well-posedness is then based on simple a priori estimates.

However, there are some important examples, for instance,  in the areas of fluid dynamics or kinetic theories, in which the advecting velocity fields and the observed quantities are rather rough. Thus, the mathematical theory for the advection-diffusion equation \eqref{eq:adv-diff} falls out of the setting covered by classical theories, and it has been discovered only recently that there are situations in which integrable distributional solutions cease to be unique \cite{ModenaSzekelyhidi18,ModenaSattig2020,CheskidovLuo2021}. We will discuss this issue later in more detail.

A customary proof of uniqueness is based on energy estimates. We are aware of two  approaches to energy estimates that deal with different weak hypotheses on regularity and integrability of vector fields and solutions, both embarking from the following (at first formal) estimates 
\begin{equation}\label{eq:energy:estimate}
\frac1{(q-1)q}\frac{d}{dt} \int_{\R^d} |\theta|^q\, dx + \kappa \int_{\R^d} |\theta|^{q-2} |\grad \theta|^2\, dx \le \begin{cases}  \int_{\R^d} |\theta|^{q-1} |\grad\theta| |u|\, dx,\\[1em]
	\int_{\R^d} |\theta|^q (\div u)^{-}\, dx,
\end{cases}\qquad\text{ for } q>1.
\end{equation}
Here, the superscript minus sign labels the negative part of the divergence. The limit $q\to 1$ leads to an a priori estimate in terms of the entropy and is discussed in Remark~\ref{R1} below.

The first approach, which is based on the first estimate, applies to velocity fields in the integrability class $ L^r((0,T);L^p(\rd))$ provided that $r$ and $p$ satisfy the so-called Ladyzhenskaya--Prodi--Serrin condition
\begin{equation*}
	\frac{2}{r}+\frac{d}{p} \leq 1.
\end{equation*} 
Here, the task is to bound the integral on the right-hand side in terms of those on the left-hand side, which can be achieved with standard H\"older and Sobolev inequalities. Apparently, the parabolic structure is of fundamental importance in this approach and the method ceases to hold in the non-diffusive setting $\kappa=0$. 
Since we are particularly interested in estimates that hold uniformly for positive but arbitrary small diffusivity parameter $\kappa$, we will not further elaborate on it here. We refer to \cite{BianchiniColomboCrippaSpinolo17} for a simple proof in the $q=2$ setting and a discussion on optimality.

The second approach is particularly important in models in which the fluid is at most weakly compressible in the sense that
\begin{equation}\label{eq:near_incompressible}
	(\nabla\cdot u)^- \in L^1\bigl((0,T);L^\infty(\rd)\bigr).
\end{equation}
In this case, an application of a Gronwall argument implies the estimate
\begin{equation}
	\|\theta\|_{L^\infty(L^q)} + c_{\kappa, q} \|\grad|\theta|^{\frac{q}2}\|_{L^2(L^2)}^{\frac2q} \lesssim \Lambda^{1-\frac{1}{q}}\|\theta^0\|_{L^q} 
	\label{eq:apriori_1}
\end{equation}
where $c_{\kappa,q}^q  = \kappa(q-1)$ and $\log \Lambda = \|(\nabla\cdot u)^-\|_{L^1(L^\infty)}$. In order to make this approach work rigorously, the validity of the chain rule has to be confirmed to establish the energy estimate~\eqref{eq:energy:estimate}. This leads us to the concept of \emph{renormalized solutions}, which were originally introduced by DiPerna and Lions in \cite{DiPernaLions89}: An integrable function $\theta$ is called a renormalized solution to the advection-diffusion equation \eqref{eq:adv-diff} if it satisfies
\begin{equation}\label{e:renomalized:solution}
	\partial_t\beta(\theta) + \nabla\cdot(u\beta(\theta)) = (\beta(\theta)-\theta\beta'(\theta))\nabla\cdot u + \kappa\Delta\beta(\theta) - \kappa \beta''(\theta) |\grad\theta|^2
\end{equation}
in the distributional sense, for any bounded $C^2$ function $\beta:\mathbb{R}\rightarrow\mathbb{R}$ whose derivatives are bounded and vanish at zero. Renormalized solutions are easily proved to be unique, and DiPerna and Lions' theory shows that distributional solutions in $L^{\infty}((0,T);L^q(\R^d))$ are renormalized if the advecting velocity field is Sobolev regular in the spatial coordinate, namely $u\in L^1 ((0,T);W^{1,p}(\R^d))$, and $p$ and $q$ have to be H\"older conjugates (or larger), $1/p+1/q\le 1$. In what follows, we will occasionally refer to this setting as the \emph{DiPerna--Lions setting}.

The DiPerna--Lions theory was further extended by Ambrosio~\cite{Ambrosio04} to vector-fields of $BV$-regularity. However, there are certain situations in which a direct verification of the renormalization property~\eqref{e:renomalized:solution} seems to fail.
One regards well-posedness results for vector-fields, whose gradient is a singular integral of an $L^1$ function~\cite{CrippaNobiliSeisSpirito17}, which is of interest in certain problems in fluid dynamics. We revisit this setting later in this paper.

It is certainly surprising that the regularizing effect of diffusion does not rule out non-uniqueness and that the DiPerna--Lions setting is \emph{both} optimal for the advection equation ($\kappa=0$ in~\eqref{eq:adv-diff}) as well as the advection-diffusion equation~\eqref{eq:adv-diff}, at least up to a dimension-dependent gap. Indeed, Modena, Sattig, and Székelyhidi \cite{ModenaSzekelyhidi18,ModenaSattig2020} showed that there are Sobolev vector fields $u\in C([0,T]; (L^p \cap W^{1,\tilde p})(\mathbb{T}^d))$ for which uniqueness fails in the class of densities $\theta\in C([0,T];L^{q}(\mathbb{T}^d))$ with $p, q\in (1,\infty)$, $\tilde p \in [1,\infty)$ and $1/p+1/q=1$ and such that it holds
\begin{equation*}
\frac{1}{\tilde p}+\frac{1}{q} > 1 + \frac{1}{d} \qquad \text{and} \qquad p < d.
\end{equation*}
An improvement on the above condition is obtained by Cheskidov and Luo~\cite{CheskidovLuo2021} for the transport equation ($\kappa=0$) at the expense of a worse time-integrability, that is solutions $\theta \in L^1([0,T];L^q(\mathbb{T}^d))$ with $1/p+1/q>1$.

DiPerna and Lions' theory is extremely powerful and finds numerous applications to various types of advection and kinetic equations. As a by-product, 
it provides qualitative stability statements for the linear equation \eqref{eq:adv-diff}. For instance,  considering two different solutions with two different advection fields,
\begin{equation*}
	\partial_t \theta_\varepsilon + \nabla\cdot(u_\varepsilon\theta_\varepsilon) = \kappa\Delta\theta_\varepsilon, \qquad \partial_t \theta + \nabla\cdot(u\theta) = \kappa\Delta\theta,
\end{equation*}
we know that $\theta_\varepsilon \rightarrow \theta$ when $\varepsilon\rightarrow 0$ provided $u_\varepsilon \rightarrow u$ in some suitable norms. Similarly,  for vanishing diffusivities, $\kappa\rightarrow 0$, solutions of the advection-diffusion equation \eqref{eq:adv-diff} converge to the solutions of the transport equation 
\begin{equation}
	\partial_t \theta + \nabla\cdot(u\theta) = 0.
	\label{eq:transport}
\end{equation}
By nature, DiPerna and Lions' theory cannot provide rates in these qualitative stability statements. Besides establishing well-posedness, in particular uniqueness, those are interesting from the point of view of an error analysis for numerical approximations. But also for modeling purposes, quantitative results are crucial, for instance, with regard to the zero-diffusivity limit $\kappa\rightarrow 0$. 

The purpose of this paper is to derive \emph{stability estimates} for the advection-diffusion equation~\eqref{eq:adv-diff} both in the DiPerna--Lions setting and the slightly more singular setting from~\cite{CrippaNobiliSeisSpirito17}.

The works \cite{Seis17,Seis18} provide a new quantitative approach to the advection equation \eqref{eq:transport} in the DiPerna--Lions setting. The approach not only rediscovers most of the results from the original paper \cite{DiPernaLions89} but also offers sharp stability estimates that extend to situations inaccessible by the traditional renormalization approach. 
A typical estimate in this context compares two distributional solutions $\theta_1,\theta_2\in L^\infty((0,T);L^q(\rd))$ of the Cauchy problem \eqref{eq:transport} corresponding to two different weakly compressible velocity fields $u_1,u_2\in L^1((0,T);L^p(\rd))$, respectively, where $1/p+1/q=1$. Then it holds,
\begin{equation}
	\sup_{t\in(0,T)}\mathcal{D}_\delta(\theta_1(t,\cdot),\theta_2(t,\cdot)) \lesssim \|\nabla u_1\|_{L^1(L^p)}\bigl(\|\theta_1\|_{L^\infty(L^q)}+\|\theta_2\|_{L^\infty(L^q)}\bigr) + 1,
	\label{eq:first_estimate}
\end{equation}
provided that $u_1\in L^1((0,T);W^{1,p}(\rd))$ and where $\delta = \|u_1-u_2\|_{L^1(L^p)}$ is the distance of the velocity fields. The quantity $\mathcal{D}_\delta(\cdot,\cdot)$ on the left-hand side is a Kantorovich--Rubinstein distance associated to a logarithmically increasing cost, originally arising in optimal transportation theory and defined as
\begin{equation}
	\mathcal{D}_\delta(\mu,\nu) = \inf_{\pi\in\Pi(\mu,\nu)} \iint_{\rd\times\rd} \log\biggl(\frac{|x-y|}{\delta} + 1 \biggr) \, d\pi(x,y).
	\label{eq:Ddelta}
\end{equation}
Here, $\mu,\nu$ are finite measures on $\rd$ such that $\mu[\rd]=\nu[\rd]$ and $\Pi(\mu,\nu)$ is the set of couplings, i.e., measures on $\rd\times \rd$ such that $\pi[A\times \rd] =\mu[A]$ and $\pi[\rd\times A]=\nu[A]$ for all measurable $A\subset\rd$. The parameter $\delta>0$ here plays a crucial role because it can be understood as the rate of convergence between the two densities $\theta_1$, $\theta_2$ in terms of some parameter (the $L^1((0,T);L^p(\rd))$ distance between $u_1$ and $u_2$ in this case). Please see Section~\ref{s:opt} for more details about this topic and related optimal transport distances. 

The first version of the stability estimate \eqref{eq:Ddelta} was introduced in \cite{BOS}, and we shall comment briefly on its origin. The fact that there are logarithmic distances appearing is not really surprising since they are already present at the level of the flow. Consider the ODE associated to the transport equation without diffusion, that is the \emph{Lagrangian setting}, as the equation for the flow $\phi_t:\rd\rightarrow\rd$ with $t\in(0,T)$,
\begin{equation}
	\left\lbrace
	\begin{array}{l}
		\partial_t\phi_t = u_t\circ \phi_t, \\
		\phi_0 = \id.
	\end{array}
	\right.
	\label{eq:ode}
\end{equation}
Then, for two solutions of the flow equation \eqref{eq:ode} there is also the following logarithmic stability estimate that can be verified straightforwardly,
\begin{equation}
	\log\left(\frac{|\phi^1_t(x)-\phi^2_t(x)|}{\delta} + 1\right) \lesssim \|\nabla u_1\|_{L^1(L^\infty)} + 1,
	\label{eq:log_ode}
\end{equation}
provided that $\delta = \|u_1-u_2\|_{L^1(L^\infty)}$ and that $u_1$ is Lipschitz.
Analogous results in the DiPerna--Lions setting for the non-diffusive case have been proved by Crippa and De Lellis \cite{CrippaDeLellis08} if $p>1$ by replacing the  uniform bounds in the space variable by suitable integral averages. Unfortunately, due to technical limitations, it is currently unclear how to extend these \emph{optimal} stability estimates to the case  $p=1$. There are, however, non-optimal extensions of the Crippa--De Lellis theory to the case $p=1$ by Jabin \cite{Jabin10} and to lower regular vector fields, namely vector fields whose gradient is given by a singular integral of an $L^1$ function, by Bouchut and Crippa \cite{BouchutCrippa13}.

The ordinary differential equation of the flow \eqref{eq:ode} can be related to the transport equation~\eqref{eq:transport} through the method of characteristics,
\begin{equation*}
	\theta(t,\cdot) = (\phi_t)_\# \theta^0,
\end{equation*}
and therefore the existence of an analogy between \eqref{eq:first_estimate} and \eqref{eq:log_ode} is not unexpected.

The stability estimates for the transport equation \cite{Seis17,Seis18} turned out to be quite flexible. There are actually many applications for which the optimal estimate plays a crucial role, for example in coarsening and mixing problems \cite{BOS,OSS,Seis13b,Seis20} or when studying error estimates for numerical approximations \cite{SchlichtingSeis17,SchlichtingSeis18}. Moreover, stability estimates are successfully extended to certain settings, in which renormalization in the sense of DiPerna and Lions could not be established directly \cite{CrippaNobiliSeisSpirito17,ClopJylhaMateuOrobitg19}. As a consequence, new well-posedness results are proven. In the stochastic setting, (sub-optimal) stability estimates were derived in \cite{LiLuo19,FathiMikulincer20}.

In the present paper, we intend to generalize these optimal stability estimates to transport equations with diffusivity, i.e., $\kappa>0$. While the regularizing effect of the diffusion might on a qualitative level indicate that estimates holding for the transport equation \eqref{eq:transport} carry over to the advection-diffusion equation \eqref{eq:adv-diff}, the adaptation of the mathematical proofs is not straightforward. As it turns out, our analysis is limited to advection-diffusion equations with constant diffusivities, while qualitative results are available for more general diffusions \cite{LeBrisLions08,Figalli08,LeBrisLions2019}.
Nonetheless, these estimates are an important contribution to the existing theory. For instance, optimal bounds on convergence rates of numerical approximations become accessible for the first time in the DiPerna--Lions setting. We are currently addressing these in an upcoming work \cite{NavarroFernandezSchlichting21}. Moreover, the new bounds are potentially applicable in the study of mixing problems in fluid dynamics, see also \cite{Seis13b,Seis20} for related work. In addition, our estimate extends the existing well-posedness theory for \eqref{eq:adv-diff} to fluids with an $L^1$ vorticity, $\grad \times u \in L^1$, which are of relevance, for instance, in the study of the $2D$ Euler and Navier--Stokes equations \cite{CrippaNobiliSeisSpirito17,NussenzveigLopesSeisWiedemann20,CiampaCrippaSpirito20}.

We finally mention that quantitative estimates on advection-diffusion equation were also obtained in the recent works \cite{BrueNguyen2021,Seis20}. These, however, focus on quantifying the vanishing diffusivity limit with applications to mixing.

\medskip

\emph{Throughout this paper we will use the following notation:} For any $A\subseteq X$ we denote its complementary as $A^c = X\setminus A$. We designate $\mathcal{L}^d$ to the Lebesgue measure on $\rd$. For any $a\in\rd$ and $r>0$, we denote $B_r(a)$ to the ball of radius $r$ centered at $a$. We write the positive and negative parts of a function as $f^+(x) = \max\{f(x),0\}\geq 0$ and $f^-(x) = -\min\{f(x),0\}\geq 0$ respectively, so that $f=f^+-f^-$ and $|f|=f^++f^-$. If a function $f$ is defined in a space $X$ and it holds $f\in B(X)$ for any $B$ Banach space, we can write without loss of generality $f\in B$. If it is unclear, however, we will explicitly write $f\in B(X)$. In addition, there will be many occasions with $X=(0,T)\times\rd$, so for any Banach spaces $B_1$ and $B_2$ we denote $f\in B_1(B_2)$ for $f\in B_1((0,T);B_2(\rd))$. For any vector function $f:X\rightarrow\rd$ that lives in a Banach space $B$, we address $f\in B(X)$ for $f\in B(X)^d$. Moreover, the symbol $\#$ means the push-forward and we write $a \lesssim b$ when there is a constant $C>0$ such that $a \leq Cb$. The constant might depend on many quantities that are not relevant for the estimate that we are dealing with.

\medskip

\emph{The article is organized as follows:} In the next section, Section \ref{S:results}, we state the precise definitions and present and discuss our main results.  The optimality of our Theorem \ref{thm:1} is questioned in Section \ref{s:optimality}. Section \ref{s:opt} provides a brief background on tools from the theory of optimal transportation used in our proofs. In Section \ref{s:main}, we present the proof of our general stability estimate in Theorem \ref{thm:1}. The final Section \ref{ss:uniqueness} contains a uniqueness result for vector fields with $L^1$ vorticites, Theorem \ref{thm:2}.

\section{Main results}\label{S:results}

Our first main result of this paper concerns a stability estimate for the advection-diffusion equation~\eqref{eq:adv-diff} using the optimal transportation distance~\eqref{eq:Ddelta}. For this, we are considering precisely the   DiPerna--Lions setting \cite{DiPernaLions89} that we introduced before, that is, for the velocity field, we impose   Sobolev regularity in the spatial variable,
\begin{equation}
	u\in L^1((0,T);W^{1,p}(\rd)) \qquad\text{ for some  }  p\in (1,\infty],
	\label{eq:dpl_setting}
\end{equation}
while for the initial datum, we demand some integrability,
\begin{equation}
	\theta^0\in (L^1\cap L^q)(\rd) \qquad \text{ with }  q>1 \mbox{ such that }\frac1p + \frac1q\le 1.
	\label{eq:initial}
\end{equation}
Working in the full space requires to suppose some additional decay properties, for instance, in order to 
ensure that the logarithmic Kantorovich--Rubinstein norms $D_{\delta}(\theta^0)$ are finite for \emph{any finite}~$\delta$. We achieve this by additionally assuming that $\theta^0$ has finite first moments,
\begin{equation}\label{100}
\intrd |x|\;|\theta^0(x)|\, dx <\infty.
\end{equation}
In this setting, it can be established that (smooth) solutions have $L^1$ temporal-spatial gradients as will be outlined in Remark \ref{R1} below. 
On bounded domains, the latter is always true as a consequence of the a priori estimates in \eqref{eq:apriori_1}.

Let us now present the precise statement.

\begin{theorem}\label{thm:1}
	Let $u_1$ and $u_2$ be two vector fields satisfying \eqref{eq:dpl_setting}, and let $\theta_1^0$ and $\theta_2^0$ be two initial data such that \eqref{eq:initial} holds. Let $\kappa_1$ and $\kappa_2$ be two positive constants. Then, for any two distributional solutions  $\theta_1,\theta_2 \in L^\infty(L^1\cap L^q)$ to the advection-diffusion equation \eqref{eq:adv-diff} corresponding to $\kappa_1$, $u_1$, $\theta_1^0$ and $\kappa_2$, $u_2$, $\theta_2^0$, respectively, which satisfy $\theta_1,\theta_2 \in L^1(W^{1,1})$, the following stability estimate holds
	\begin{equation}
		\sup_{0\leq t\leq T}\mathcal{D}_\delta(\theta_1(t,\cdot),\theta_2(t,\cdot)) \lesssim \mathcal{D}_\delta(\theta^0_1,\theta^0_2) + 1 + \frac{\|u_1-u_2\|_{L^1(L^p)}+|\kappa_1-\kappa_2| \;\|\nabla\theta_2\|_{L^1(L^1)}}{\delta}.
		\label{eq:stab_est}
	\end{equation}
\end{theorem}

Our estimate is optimal in many regards, as will be discussed in some detail in Section \ref{s:optimality} below.

It is not difficult to see that the stability estimate implies uniqueness. Indeed, suppose $\theta_1$ and $\theta_2 $ are two solutions to the advection-diffusion equation with the same velocity, initial data, and diffusivity constant. In that case, the right-hand side becomes independent of $\delta$ and letting $\delta \to 0$; the left-hand side would blow up except if both solutions are identical. The argument could be made rigorous, for instance, by a straightforward application of Lemma \ref{lemma:bound_otd} below, and we will elaborate on this principle in the proof of Theorem~\ref{thm:2}.
We thus recover the uniqueness results for distributional solutions from the original paper by DiPerna and Lions \cite{DiPernaLions89} in a new quantitative way.

The implicit constant in the stability estimate \eqref{eq:stab_est} depends on the $L^{\infty}(L^1)$ and $L^{\infty}(L^q)$ norms of both solutions, but only on the $L^1(L^p)$ norm of one of the velocity gradients and the $L^1(L^1)$ norm of one of the solutions gradients. Consequently, it would be enough to assume such regularity for one of the vector fields and one of the solutions, and we would get an estimate on the distance between the unique solution and a non-unique approximant. Furthermore, our analysis applies also to the situation where one of the diffusivity constants depends on $x$, in which case the modulus of the difference of the diffusivity constants needs to be replaced by $\|\kappa_1 - \kappa_2(x)\|_{L^{\infty}}$. 

We finally remark on our hypotheses in Theorem \ref{thm:1}. For initial data in \eqref{eq:initial} and velocity fields in~\eqref{eq:dpl_setting}, the existence of distributional solutions in $L^{\infty}(L^1\cap L^{q})$ is easily established with the help of the a priori estimates in \eqref{eq:apriori_1} via smooth approximation, \emph{provided} that the velocity field satisfies the weak compressibility condition in \eqref{eq:near_incompressible}. In the following remark, we comment on the gradient condition.
\begin{remark}\label{R1}
The regularity assumption $\grad \theta \in L^1(L^1)$ that is assumed in Theorem \ref{thm:1} is satisfied for finite entropy solutions as long as the velocity verifies the weak compressibility condition in \eqref{eq:near_incompressible}. Indeed, if $\theta$ is a nonnegative solution with
\begin{equation}\label{101}
\intrd \theta(t,x) \log \theta(t,x) \, dx \in \R\qquad \text{ for all } t\in [0,T],
\end{equation}
a standard computation reveals that
\[
\intrd \theta(t) \log \theta(t) \, dx + \kappa \int_0^t \intrd \frac{|\grad \theta|^2}{\theta}\, dxdt \le \intrd \theta^0 \log\theta^0\, dx +\|(\div u)^-\|_{L^1(L^{\infty})} \|\theta\|_{L^{\infty}(L^1)},
\]
and thus, the Fisher information $\|\theta^{-1}|\grad\theta|^2\|_{L^1}$ is integrable in time. Moreover, by H\"older's inequality, we obtain
\[
\int_0^t  \|\grad\theta\|_{L^1}^2 dt \le \int_0^t \|\theta\|_{L^1} \|\theta^{-1}|\grad\theta|^2\|_{L^1}dt \le \|\theta\|_{L^{\infty}(L^1)}  \int_0^t \|\theta^{-1}|\grad\theta|^2\|_{L^1}dt ,
\] 
which is finite by \eqref{eq:apriori_1}, \eqref{101}, and the above estimate on the entropy. We easily deduce that $\grad\theta \in L^1(L^1)$ on finite time intervals, such that the last term in the right hand side of~\eqref{eq:stab_est} can be estimated by
\[
 \frac{|\kappa_1-\kappa_2| \;\|\nabla\theta_2\|_{L^1(L^1)}}{\delta} \lesssim \frac{|\kappa_1-\kappa_2|}{\delta \, \kappa_2} .
\]
It remains to understand that in the setting \eqref{eq:near_incompressible}, \eqref{eq:dpl_setting}, \eqref{eq:initial}, \eqref{100}, solutions do indeed have finite entropy for finite times \eqref{101}. 
An upper bound is provided by the elementary estimate $r\log r\lesssim r + r^q$ for any $r>0$ and the integrability assumptions on $\theta$. For the lower bound, we first notice that the moment bound in \eqref{100} is propagated in time. In fact, since $\theta^0\in L^1$ by assumption, the homogeneous weight $|x| $ in \eqref{100} can be replaced by the smoother $\sqrt{1+|x|^2}$, and we have the estimate
\[
\intrd \sqrt{1+|x|^2} \theta(t,x)\, dx \lesssim \intrd \sqrt{1+|x|^2} \theta^0(x)\, dx + \|u\|_{L^1(L^p)}\|\theta\|_{L^{\infty}(L^q)} + \kappa \|\theta\|_{L^1}.
\]
Now, since $r\log r\gtrsim - r^{\alpha}$ for any $\alpha\in (0,1)$ and any $r>0$, we conclude that the lower bound
\[
\intrd \theta \log \theta\, dx \gtrsim - \intrd\theta^{\alpha}\, dx \gtrsim -\left(\intrd \frac1{\sqrt{1+|x|^2}^{\frac{\alpha}{\alpha-1}}} \, dx\right)^{1-\alpha} \left(\intrd \sqrt{1+|x|^2}\, \theta\, dx\right)^{\alpha}
\]
is finite as long as $\alpha >\frac{d}{d+1}$.
\end{remark}
We complete the discussion of Theorem \ref{thm:1} with a comment on the integrability restriction on the velocity gradient in~\eqref{eq:dpl_setting}. A crucial step in the derivation of the stability estimate  is controlling the advection term via  an argument that was introduced by Crippa and De Lellis in \cite{CrippaDeLellis08}. The argument makes use of the Hardy--Littlewood maximal function and exploits its continuity in $L^p$ spaces for $p>1$. The restriction in \eqref{eq:dpl_setting} relies precisely on this limitation. See Section \ref{s:main} for details.

In our second theorem, that we shall motivate in the following, we use a suitable extension of the Crippa--De Lellis method to vector fields whose gradient is given by a singular integral of an $L^1$ function. A typical example of such a vector field is the velocity field that is obtained from an $L^1$ vorticity with the help of the Biot--Savart law.

The precise setting is as follows. We assume that the velocity components $u_1,\dots,u_d$ have  kernel representations,
\begin{equation}
	u_i = k_i \ast\omega_i , 	\label{eq:Hu2}
\end{equation}
for any $i\in\{1,\dots,d\}$,  where the generalized vorticity components are merely integrable,
\begin{equation}\label{103}
\omega_i\in L^1((0,T);L^1(\R^d)),
\end{equation}
and where the kernels $k_i$ are such that any of its derivatives $\partial_j k_i$ is a \emph{singular kernel}. More precisely, we suppose that \begin{enumerate}[label={\textup{(k$_{\arabic*}$)}}]
	\item\label{ass:k1} $k\in \mathcal{S}'(\rd)$, where $\mathcal{S}'(\rd)$ is the dual of the Schwartz space;
	\item $k|_{\rd\setminus\{0\}}\in C^2(\rd\setminus\{0\})$;
	\item for $\alpha\in \mathbb{N}_0^d$ with $|\alpha|\leq 2$ it holds
	\begin{equation*}
		|D^\alpha k(x)| \lesssim \frac{1}{|x|^{d-1+|\alpha|}}, \qquad \forall x\neq 0 ;
	\end{equation*}
	\item\label{ass:k4} it holds
	\begin{equation*}
		\left| \int_{R_1<|x|<R_2}\nabla k(x)dx \right| \lesssim 1 \qquad \text{for every } 0<R_1<R_2<\infty.
	\end{equation*}
\end{enumerate}
Under these conditions, the velocity field only lives in a \emph{weak} Lebesgue space globally in $\rd$,  see Lemma \ref{L100} below. For mathematical convenience, we will enforce the slightly stronger condition
\begin{equation}
	u\in L^{p,\infty}((0,T)\times\rd),
	\label{eq:Hu1}
\end{equation}
for some $p>1$. For the above mentioned applications in fluid dynamics, such a condition is always satisfied. We recall that  weak Lebesgue spaces $L^{p,\infty}$ can be defined on a measure space $(X,\mu)$ via the quasi-norm
\begin{equation}
	\|f\|_{L^{p,\infty}(\mu)} = \sup_{\lambda>0} \left( \lambda^p\mu(\{x\in X:|f(x)|>\lambda\}) \right)^{1/p},
	\label{eq:weakLp}
\end{equation}
for any measurable $f:X\rightarrow \mathbb{R}$. In the case $p=\infty$, we adopt the convention $L^{\infty,\infty} = L^\infty$. Notice that $\|\cdot\|_{L^{p,\infty}}$ is not a norm since it does not verify the triangle inequality. Also recall that there is an embedding $L^p\subset L^{p,\infty}$ with $\|f\|_{L^{p,\infty}} \leq \|f\|_{L^p}$ for every $f\in L^p$.

As our goal is to derive a full well-posedness theorem, and not only a uniqueness result, we shall, in addition, assume that the velocity field satisfies the weak compressibility condition \eqref{eq:near_incompressible}.

Finally, we work with with initial datum that are integrable and bounded,
\begin{equation}
	\theta^0 \in (L^1\cap L^\infty)(\rd),
	\label{eq:Hid}
\end{equation} 
and that have finite first moments, i.e., \eqref{100} holds.

With this list of assumptions, we can state our second main result.
\begin{theorem}\label{thm:2}
	Let $u$ be a velocity field satisfying \eqref{eq:Hu2}, \eqref{103}, \eqref{eq:Hu1}, and \eqref{eq:near_incompressible}, and let $\theta^0$ be an initial datum verifying \eqref{eq:Hid} and~\eqref{100}. Then the Cauchy problem \eqref{eq:adv-diff} has a unique distributional solution $\theta(t,x)$ in the class $L^\infty((0,T);(L^1\cap L^\infty)(\rd))$ with $\grad \theta\in L^1((0,T)\times \R^d)$.
	\label{thm:exist_uniq}
\end{theorem}
A corresponding  result  on uniqueness for the transport equation \eqref{eq:transport} was previously derived in \cite{CrippaNobiliSeisSpirito17}, which in turn builds up on the Lagrangian setting considered in \cite{BouchutCrippa13}. Here, we develop an analogous theory for the diffusive case $\kappa>0$. The result is, of course, not unexpected since the diffusive equation is usually considered to produce even smoother solutions. However, we are currently not aware of any   techniques, apart from those developed here, in which such a result can be established. Moreover, we are presently unable to produce results in more general settings, for instance, for non-constant diffusivities. 

The new uniqueness result in Theorem \ref{thm:2} has potential applications in the study of the inviscid limit for the two-dimensional Navier--Stokes equations with rough forcing. We believe that thanks to  our present contribution, the recent results in \cite{NussenzveigLopesSeisWiedemann20,CiampaCrippaSpirito20} can be extended to such an interesting setting.

Finally, in Theorem \ref{thm:2} we add a proof of existence that is based on classical techniques. Assumptions on the advection field itself and the  weak compressibility assumption are only used to establish the existence part. 

\section{On the optimality of Theorem \ref{thm:1}}\label{s:optimality}

Because Kantorovich--Rubinstein distances metrize weak convergence, cf.~Theorem 7.12 in~\cite{Villani03}, 
the result of Theorem \ref{thm:1},
\[
		\sup_{0\leq t\leq T}\mathcal{D}_\delta(\theta(t,\cdot),\theta_n(t,\cdot)) \lesssim \mathcal{D}_\delta(\theta^0,\theta^0_n) + 1 + \frac{\|u-u_n\|_{L^1(L^p)}+|\kappa-\kappa_n|  }{\delta},
\]
cf.~\eqref{eq:stab_est}, can be considered as an estimate on the rates of weak convergence for the advection-diffusion equation for three types of approximations: The convergence of solution sequences corresponding to converging sequences of initial data $\theta^0_n \to \theta^0$, velocity fields $u_n\to u$ and diffusivity constants $\kappa_n\to \kappa$. Here, an interesting feature is that the rate of convergence is already incorporated into the distance function, and it is given by the smallest $\delta=\delta_n$ for which the right-hand side is finite, i.e.,
\[
D_{\delta_n}(\theta^0,\theta_n^0) \lesssim 1, \qquad \|u-u_n\|_{L^1(L^p)}\lesssim \delta_n,\qquad |\kappa-\kappa_n|\lesssim\delta_n.
\]

In order to discuss our convergence result for each of the coefficients individually, let us decompose the setting into three separated problems.

\medskip

\emph{Perturbations of initial data.}  We start with the situation in which $\theta_n^0\to \theta^0$, while $u_n=u$ and $\kappa_n=\kappa$. In this case, the a priori estimate \eqref{eq:apriori_1} entails that the respective solutions converge strongly if the initial data do, more precisely,
\[
\|\theta_n -\theta\|_{L^{\infty}(L^q)} \lesssim \|\theta_n^0-\theta^0\|_{L^q}.
\]
This estimate is certainly optimal, but it falls short if the convergence of the initial data needs to be measured with respect to weaker topologies, as it is the case in the analysis of numerical approximation schemes for the problem at hand. Indeed, if $\theta^0_n$ denotes a discretization of a merely integrable initial datum $\theta^0$ on a mesh of size $h=1/n$, the approximation is converging weakly with rate $1/n$, see, e.g., \cite{SchlichtingSeis17,SchlichtingSeis18} in the case of a finite volume approximation, but there is no rate of convergence in any strong Lebesgue norm. Therefore, in order to estimate the approximation error for numerical schemes for the advection-diffusion equation in the DiPerna--Lions setting, it is necessary to choose distances which metrize weak convergence such a negative Sobolev norms or Kantorovich--Rubinstein distances. For us, this is the main motivation for the derivation of Theorem \ref{thm:1}. A numerical analysis of finite volume schemes for the model under consideration is currently prepared by two of us \cite{NavarroFernandezSchlichting21}.

\medskip

\emph{Perturbations of the velocity field.} Let us now consider the case in which $u_n\to u$, while $\theta_n^0=\theta^0$ and $\kappa_n=\kappa$. We start with commenting on the non-diffusive setting $\kappa=0$ that was investigated in \cite{DeLellisGwiazdaSwierczewska16,Seis17,Seis18}. Here, a simple example of an oscillatory flow  produces oscillatory solutions which cannot converge strongly, and whose optimal order of weak convergence has been established in terms of the Kantorovich--Rubinstein distance $\mathcal{D}_{\delta}$.

In the presence of diffusion, $\kappa>0$, the situation is different as the regularizing effect of the parabolicity becomes dominant on any scale. As a consequence, whenever $u_n\to u$ in $L^1(L^p)$ and $u$ belongs to the DiPerna--Lions setting \eqref{eq:dpl_setting}, solutions do converge strongly in some Lebesgue norm, which can be verified by a straightforward application of the Aubin--Lions lemma.

It is currently unknown whether rates of strong convergence can be established in the general DiPerna--Lions setting. For vector fields belonging to the Ladyzhenskaya--Prodi--Serrin class, this is certainly true as can be seen by inspecting, for instance,  the analysis in \cite{BianchiniColomboCrippaSpinolo17}. It seems to us that outside of this class, rates of convergence can in general not exist. A rigorous proof of this conjecture, however, remains still open at this point. In view of the application of Theorem~\ref{thm:1} we have in mind, cf.~\cite{NavarroFernandezSchlichting21}, this question is of course redundant as  both velocity field and initial datum are approximated simultaneously and no strong rates of convergence can be expected for the latter approximations as outlined above.
Our forthcoming work~\cite{NavarroFernandezSchlichting21} moreover shows that the rate in~\eqref{eq:stab_est} is optimal, if $u_{n} = u_{1/h}$ is a finite volume discretization of a Sobolev vector field~$u$ on a mesh of size $h$. 

\medskip

\emph{Perturbations of diffusivity constant.} We finally consider the problem in which $\kappa_n\to \kappa$ while $\theta_n^0=\theta^0$ and $u_n=u$. At first glance, this problem seems to be purely academic, at least as long as we are not investigating the vanishing viscosity limit $\kappa_n\to0$. The latter has been studied already in \cite{Seis17,Seis18}, and it has been showed that solutions to this problem converge strongly but with no rate, while the optimal rate of weak convergence is of the order $\sqrt{\kappa}$. This limit also plays a role in the analysis of numerical approximations for the purely advective equation, $\kappa=0$, see \cite{SchlichtingSeis17,SchlichtingSeis18}.

The problem for finite diffusivity constants $\kappa_n\to\kappa>0$ is of importance in the analysis of numerical schemes for  advection-diffusion equations, where the discretization of the Laplacian produces an error that is, at least heuristically, related to a perturbation of the diffusivity.  

Stability estimates for this type of perturbations were derived earlier by Li and Luo \cite{LiLuo19}, building on the stability theory for advection equations developed in \cite{Seis17,Seis18}. In this work, the authors consider Lagrangian solutions and find only the suboptimal $\sqrt{|\kappa_n-\kappa|}$ error by treating the diffusion as a perturbation to the advection.

In the present paper, our objective is to improve  the rate of weak convergence from $\sqrt{|\kappa_n-\kappa|}$ to $|\kappa_n-\kappa|$, in order to estimate optimally the numerical error induced by finite volume schemes for the advection-diffusion equation, cf.~\cite{NavarroFernandezSchlichting21}. The improvement from ``exponent $1/2$ convergence'' to ``exponent $1$ convergence'' is well-known by numerical analysts who study numerical schemes for advection-diffusion equations  in a more regular setting, see e.g.~\cite[Chapter 4]{EymardGallouetHerbin00}. It is a consequence of the smoothing effect of diffusion, which is best understood if we neglect the advection for a moment: In this case,  the zero-diffusivity limit corresponds to the convergence of smooth solutions to its \emph{rough} initial configuration, while in the case of positive diffusivity, \emph{smooth} solutions are approximated. 

Simple calculations show that under our general assumptions, the convergence of the corresponding solutions takes place in strong topologies --- this follows again from the Aubin--Lions lemma. Regarding the rate of weak convergence, the following simple example vaguely indicates that our findings are optimal. For simplicity, we consider the one-dimensional case only and a Dirac function as  initial datum. Moreover, we trade the logarithmic Kantorovich--Rubinstein distance for the Wasserstein distance $W_1$, see \eqref{eq:wasserstein} below, because computations here can be made more explicit. 
\begin{example}\label{prop:obtimality}
	Suppose that $\theta_1$ and $\theta_2$ denote the one-dimensional heat kernels corresponding to diffusivities $\kappa_1>\kappa_2>0$. Then it holds
	\begin{equation*}
		\frac{t |\kappa_1-\kappa_2|}{\sqrt{\kappa_1}+\sqrt{\kappa_2}}\lesssim W_1\bigl(\theta_1(t,\cdot),\theta_2(t,\cdot)\bigr).
	\end{equation*}
Indeed, since the Wasserstein distance is the supremum over all Lipschitz functions $\psi(x)$ with Lipschitz constant bounded by $1$, cf.~\eqref{eq:wasserstein} below, we can consider the specific function $\psi(x) = |x|$ to produce an explicit lower bound. Then, by means of the change of variables $x = 2\sqrt{t\kappa_i}y$ for $i=1,2$ we obtain
\begin{equation*}
	\begin{split}
		W_1(\theta_1(t,\cdot),\theta_2(t,\cdot)) & \geq \intrd(\theta_1(t,x)-\theta_2(t,x))\psi(x)\, dx \\
		& = \frac{1}{(4\pi\kappa_1 t)^{d/2}}\intrd |x|e^{-\frac{|x|^2}{4\kappa_1 t}}dx - \frac{1}{(4\pi\kappa_2 t)^{d/2}}\intrd |x| e^{-\frac{|x|^2}{4\kappa_2 t}}dx \\
		& = \frac{2\sqrt{t}}{\pi^{d/2}}(\sqrt{\kappa_1}-\sqrt{\kappa_2})\intrd |y|e^{-|y|^2}dy \gtrsim \frac{\sqrt{t}}{\sqrt{\kappa_1}+\sqrt{\kappa_2}}(\kappa_1-\kappa_2).
	\end{split}
\end{equation*}
\end{example}
Of course, despite the fact that the Kantorovich--Rubinstein distance $\mathcal{D}_{\delta}$ and the Wasserstein distance $W_1$ both metrize weak convergence, we are aware of the fact that they are not equivalent and thus  the actual rate of convergence might be different. 
We have the one-sided estimate $\mathcal{D}_{\delta} \le W_1/\delta$ which follows from linearizing the logarithm. However, the opposite estimate, that would be desirable here, does not hold true. Moreover, also by direct calculations, is not clear to us how to extend the example to the $\mathcal{D}_{\delta}$ distance.
Nonetheless, we rush ahead with the bold assertion that our findings are optimal \emph{morally}.

\section{Optimal transportation distances and distributional solutions}
\label{s:opt}

In this section we introduce some tools from the theory of optimal transportation that will be useful in the poofs.
In particular, we show how the distance~\eqref{eq:Ddelta} and related distances change under the advection-diffusion equation~\eqref{eq:adv-diff} (see Lemma~\ref{lem:dtD}).
We decide to make a presentation suitable for our needs, in a rather smooth setting that is enough for our purposes. 
For generalizations and detailed proofs of the subsequent results, we refer to Villani's monograph~\cite{Villani03}.

We consider the transport of nonnegative densities $\mu,\nu \in L^1_+(\rd)=\{u\in L^1(\rd): u\geq 0\}$. We recall that $\Pi(\mu,\nu)$ denotes the set of all transport plans between two densities $\mu$ and $\nu$ with $\mu[\rd]=\nu[\rd]\in \mathbb{R}$. That is, $\pi\in\Pi(\mu,\nu)$ if it is a measure on $\rd\times\rd$ such that $\pi[A\times\rd] = \mu[\rd]$ and $\pi[\rd\times A] = \nu[A]$ for all measurable sets $A\subseteq\rd$. 
The cost for transporting mass over a distance $z$ is modeled by a continuous nondecreasing cost function $c:[0,\infty) \to [0,\infty)$. The general optimal transportation problem, or \emph{Kantorovich problem}, consists of attaining an optimal transport plan that minimizes the total transportation cost. Therefore the minimal transportation cost is given by
\begin{equation}
\mathcal{D}_c(\mu,\nu) = \inf_{\pi\in\Pi(\mu,\nu)} \iint_{\rd\times\rd} c(|x-y|)d\pi(x,y).
\label{eq:ot_primal}
\end{equation}
Physically or economically speaking we could say that $\mathcal{D}_c(\mu,\nu)$ measures the minimal total cost of transporting an initial configuration of mass or goods given by $\mu$ to a final configuration $\nu$ if the cost of the transport of an infinitesimal part is modelled by $c$.

Here we will consider cost functions $c:[0,\infty)\rightarrow[0,\infty)$ that are strictly concave, Lipschitz with uniform Lipschitz constant $L_c$ and such that $c(0)=0$. This type of cost functions induces a metric $d(x,y) = c(|x-y|)$ on $\rd$. Moreover, as it is proved in \cite[Theorem 1.14]{Villani03}, in this case the optimal transportation problem \eqref{eq:ot_primal} admits a dual formulation,
\begin{equation}
\mathcal{D}_c(\mu,\nu) = \sup_{\zeta:\rd\rightarrow \mathbb{R}} \left\lbrace \intrd (\mu(x)-\nu(x))\zeta(x) dx : |\zeta(x)-\zeta(y)| \leq c(|x-y|) \right\rbrace.
\label{eq:ot_dual}
\end{equation}
Note then that $\mathcal{D}_c(\mu,\nu)$ is a transshipment cost that only depends on the difference $\mu-\nu$.
This allows us to extend the theory to densities that are not necessarily nonnegative but just that verify $\mu[\R^d]=\nu[\R^d]\in \R$. 
In addition, because $d(x,y) = c(|x-y|)$ is a metric on $\rd$, $\mathcal{D}_c(\cdot,\cdot)$ defines a metric on the space of densities with the same total mass and it is usually referred to as the Kantorovich--Rubinstein distance or, more generally, optimal transportation distance. Therefore, for any function $\theta\in L^1(\rd)$ with zero total mass, $\theta^+[\rd]=\theta^-[\R^d]$, we conveniently define the norm
\begin{equation}
\mathcal{D}_c(\theta) =  \mathcal{D}_c(\theta^+,\theta^-).
\label{eq:ot_norm}
\end{equation}
Recall that the first problem \eqref{eq:ot_primal} admits a unique minimizer, in general $\pi_{\opt}\in\Pi(\mu,\nu)$, named the optimal transport plan. The dual problem \eqref{eq:ot_dual} also admits a maximizer, that could be nonunique $\zeta_{\opt}$, called the Kantorovich potential. It is characterized by the relation
\begin{equation*}
\zeta_{\opt}(x)-\zeta_{\opt}(y) = c(|x-y|) \qquad \text{for } d\pi_{\opt}-\text{almost all } (x,y).
\end{equation*}
We can weakly-differentiate this identity to obtain
\begin{equation*}
\nabla_x\zeta_{\opt}(x)=\nabla_y\zeta_{\opt}(y) = c'(|x-y|)\frac{x-y}{|x-y|} \qquad \text{for } d\pi_{\opt}-\text{almost all } (x,y),
\end{equation*}
and therefore it holds $|\nabla\zeta_{\opt}|\leq L_c$, since the cost function is $L_c-$Lipschitz.

There is nonetheless an additional way of presenting the optimal transport plan $\pi_{\opt}\in\Pi(\mu,\nu)$ when the cost function is strictly concave. Gangbo and McCann \cite{GangboMcCann96} proved that there exist measurable maps $T,S:\rd\rightarrow\rd$ such that it holds
\begin{equation*}
\pi_{\opt} = (\id\times T)_\# \mu = (S\times \id)_\# \nu
\end{equation*}
and where $T$ and $S$ obey the relations $\mu = S_\# \nu$, $\nu = T_\# \mu$. This characterization will be useful in the subsequent sections to prove some relevant results.

A particular example is the choice $c(z)=z$, which leads to the so-called $1$-Wasserstein distance~$W_1$, which thanks to the duality~\eqref{eq:ot_dual} can be represented in terms of $1$-Lipschitz functions as
\begin{equation}\label{eq:wasserstein}
W_1(\mu,\nu) = \sup_{\|\nabla\zeta\|_{L^\infty}\leq 1} \intrd (\mu(x)-\nu(x))\zeta(x)dx.
\end{equation}
In this paper we will consider logarithmic and bounded cost functions for the optimal transportation distance~$\mathcal{D}_c(\cdot,\cdot)$. 
First, for any $\delta>0$, we define the cost
\begin{equation}
c(z) = \log\left(\frac{z}{\delta}+1\right),
\label{eq:ot_dist}
\end{equation}
and we denote the associated Kantorovich--Rubinstein distance by $\mathcal{D}_{\delta}(\cdot,\cdot)$.
The logarithmic cost function is such that their Kantorovich potentials satisfy
\begin{equation*}
\|\nabla\zeta_{\opt}\|_{L^\infty} \leq \frac{1}{\delta}.
\end{equation*}
Secondly, we define another optimal transportation distance that will consist of a strictly concave and bounded modification of the Wasserstein distance, that is, the distance associated to the cost
\begin{equation}
c(z) = \tanh(z).
\label{eq:ot_wass_b}
\end{equation}
This distance will be denoted by $\mathcal{D}^b(\cdot,\cdot)$.

Recall that by \eqref{eq:ot_norm} both distances induce norms in the space of zero average densities. 

For every $\theta\in L^1(\rd)$ such that $\theta[\rd]=0$, $\mathcal{D}^b(\theta)$ can be controlled by~$\mathcal{D}_\delta(\theta)$ through the following Lemma, introduced and first proved in \cite{Seis17} and later adapted in \cite{CrippaNobiliSeisSpirito17} to a  framework that is similar  to ours.
\begin{lemma}[{\cite[Lemma 3.1]{CrippaNobiliSeisSpirito17}}]\label{lemma:bound_otd}
	Let $\theta\in L^1(\rd)$ be such that $\intrd\theta dx\equiv 0$. Then $\forall\gamma,\delta > 0$ it holds
	\begin{equation*}
		\mathcal{D}^b(\theta) \leq \frac{\mathcal{D}_\delta(\theta)}{\log\frac{1}{\gamma}} + \frac{\delta}{\gamma}\|\theta\|_{L^1}.
	\end{equation*}
\end{lemma}
In the forthcoming sections we will apply these optimal transportation distances with densities which depend not only on $x\in\rd$ but also on $t\in (0,T)$. Therefore the optimal transport plans or the Kantorovich potentials might be time-dependent. In order to simplify the notation we will refer as $\pi_t$ to the 
optimal transport plan $\pi_{\opt}$ associated to  the distance $\mathcal{D}_c(\mu(t,\cdot),\nu(t,\cdot))$, with $t\in (0,T)$. Analogously we write $\zeta_
t$ to denote the Kantorovich potential $\zeta_{\opt}$ associated to the same distance.

Before proceeding with some properties of the optimal transportation distances let us recall the key concept of \emph{distributional solutions}.
\begin{definition}[Distributional solutions]
	Let $u\in L^1((0,T);L^p_{loc}(\rd))$ and $\theta^0\in L^q_{loc}(\rd)$ be given for some $p$ and $q$ such that $1/p+1/q\leq 1$. A function $\theta:(0,T)\times \rd\rightarrow \mathbb{R}$ is called a distributional solution of \eqref{eq:adv-diff} with initial datum $\theta^0$ if $\theta\in L^\infty((0,T);L^q_{loc}(\rd))$ and
	\begin{equation*}
		\iint_{(0,T)\times \rd} \theta\; \bigl(\partial_t\phi + u\cdot \nabla\phi + \kappa\Delta\phi\bigr)\, dx\,dt + \int_{\rd} \theta^0\; \phi|_{t=0}\, dx = 0
	\end{equation*}
	holds for all $\phi\in C_c^\infty([0,T)\times\rd)$.
	\label{def:distr_sol}
\end{definition}
In order to obtain stability for optimal transport distances, we need to introduce a key result regarding the differentiability properties of these distances as functions of time. In this regard we will proof the result for every cost function satisfying the assumptions of Section \ref{s:opt}, since we will need it in further sections with different cost functions.
\begin{lemma}\label{lem:dtD}
	Let $\mu_1$ and $\mu_2$ be two distributional solutions in $L^1(W^{1,1})$ of the advection-diffusion equation \eqref{eq:adv-diff} with advection fields $u_1,u_2$ and diffusion coefficients $\kappa_1,\kappa_2>0$ respectively. Then the mapping $t\mapsto \mathcal{D}_c(\mu_1(t,\cdot), \mu_2(t,\cdot))$ is absolutely continuous with
	\begin{equation}\label{eq:dtD}
	\begin{split}
		\frac{d}{dt}\mathcal{D}_c(\mu_1(t,\cdot), \mu_2(t,\cdot)) & = \intrd \nabla\zeta_t\cdot (u_1(t,x)\mu_1(t,x)-u_2(t,x)\mu_2(t,x))dx \\
		& \quad - \intrd \nabla\zeta_t \cdot(\kappa_1\nabla\mu_1(t,x)-\kappa_2\nabla\mu_2(t,x)) dx
	\end{split}
	\end{equation}
	where $\zeta_t$ is the Kantorovich potential corresponding to $\mathcal{D}_c(\mu_1(t,\cdot), \mu_2(t,\cdot))$.
	\label{lemma:dtD}
\end{lemma}
\begin{proof}
	By the definition of distributional solution, by a standard approximation argument and integrating by parts, for any $h\in\mathbb{R}$ such that $t-h\in(0,T)$ we have that 
	\begin{align}
			\intrd \zeta (\mu_i(t,x)-\mu_i(t-h,x))dx & = \intrd \nabla\zeta\cdot  \int_{t-h}^t u_i(s,x)\mu_i(s,x)ds dx  + \kappa_i \intrd \zeta \int_{t-h}^t \Delta \mu_i(s,x)ds dx \nonumber \\
			& = \intrd \nabla\zeta\cdot\int_{t-h}^t u_i(s,x)\mu_i(s,x)ds dx \nonumber \\
			& \quad - \kappa_i \intrd \nabla\zeta\cdot \int_{t-h}^t \nabla \mu_i(s,x)ds dx
		\label{eq:der_D_1}
	\end{align}
	for all $\zeta\in C_c^\infty(\rd)$, almost every $t\in(0,T)$ and both $i\in\{1,2\}$. Since now $\mu_i$ and $u_i\mu_i$ are in $L^1(L^1)$ we are allowed to consider \eqref{eq:der_D_1} for all $\zeta\in W^{1,1}(\rd)$.
	
	First we will show that the mapping $t\mapsto \mathcal{D}_c(\mu_1(t,\cdot),\mu_2(t,\cdot))$ is absolutely continuous and therefore differentiable almost everywhere in $(0,T)$. By the optimality of the Kantorovich potential~$\zeta_t$ at time $t$ it holds for almost every $t\in(0,T)$ that
	\begin{align}
 		\MoveEqLeft{\mathcal{D}_c(\mu_1(t,\cdot),\mu_2(t,\cdot)) -\mathcal{D}_c(\mu_1(t-h,\cdot),\mu_2(t-h,\cdot))} \nonumber\\
		 &\leq \intrd \zeta_t (\mu_1(t,x)-\mu_1(t-h,x))dx - \intrd \zeta_t (\mu_2(t,x)-\mu_2(t-h,x))dx \nonumber \\
			& = \intrd \nabla\zeta_t\cdot  \int_{t-h}^t u_1(s,x)\mu_1(s,x)ds dx - \intrd \nabla\zeta_t\cdot  \int_{t-h}^t u_2(s,x)\mu_2(s,x)ds dx \nonumber \\
			& \quad - \kappa_1 \intrd \nabla\zeta_t\cdot \int_{t-h}^t \nabla \mu_1(s,x)ds dx + \kappa_2 \intrd \nabla\zeta_t\cdot \int_{t-h}^t \nabla \mu_2(s,x)ds dx .
		\label{eq:der_D_2}
	\end{align}
    Analogously, again by the optimality of the Kantorovich potential $\zeta_{t-h}$ at time $t-h$, it holds for almost every $t\in(0,T)$ that
	\begin{align}
			\MoveEqLeft{\mathcal{D}_c(\mu_1(t,\cdot),\mu_2(t,\cdot)) -\mathcal{D}_c(\mu_1(t-h,\cdot),\mu_2(t-h,\cdot)) } \nonumber \\
			&\leq \intrd \zeta_{t-h} (\mu_1(t,x)-\mu_1(t-h,x))dx - \intrd \zeta_{t-h} (\mu_2(t,x)-\mu_2(t-h,x))dx \nonumber \\
			& = \intrd \nabla\zeta_{t-h}\cdot  \int_{t-h}^t u_1(s,x)\mu_1(s,x)ds dx - \intrd \nabla\zeta_{t-h}\cdot  \int_{t-h}^t u_2(s,x)\mu_2(s,x)ds dx \nonumber \\
			& \quad - \kappa_1 \intrd \nabla\zeta_{t-h}\cdot \int_{t-h}^t \nabla \mu_1(s,x)ds dx + \kappa_2 \intrd \nabla\zeta_{t-h}\cdot \int_{t-h}^t \nabla \mu_2(s,x)ds dx .
		\label{eq:der_D_3}
	\end{align}
	Therefore, using that $\zeta_t$ is Lipschitz with Lipschitz constant uniformly bounded by $L_c$, we can combine \eqref{eq:der_D_2} and \eqref{eq:der_D_3} to obtain
	\begin{equation*}
	\begin{split}
	\MoveEqLeft{|\mathcal{D}_c(\mu_1(t,\cdot),\mu_2(t,\cdot))  -\mathcal{D}_c(\mu_1(t-h,\cdot),\mu_2(t-h,\cdot))| }\\
	&\leq L_c  \int_{t-h}^t \intrd |u_1(s,x)\mu_1(s,x)|dxds  + L_c  \int_{t-h}^t \intrd |u_2(s,x)\mu_2(s,x)|dxds \\
	& \quad +  L_c\kappa_1\int_{t-h}^t \intrd |\nabla\mu_1(s,x)|dxds +  L_c\kappa_2\int_{t-h}^t \intrd |\nabla\mu_2(s,x)|dxds
	\end{split}
	\end{equation*}
	for almost every $t\in(0,T)$. Since $u_i\mu_i\in L^1(L^1)$ and $\nabla\mu_i\in L^1(L^1)$ for $i\in\{1,2\}$, we conclude that indeed $t\mapsto\mathcal{D}_c(\theta(t,\cdot))$ is an absolutely continuous mapping.
	
	It remains to prove that the  derivative of the mapping takes the expression~\eqref{eq:dtD}. In order to do that, we consider again \eqref{eq:der_D_2} and \eqref{eq:der_D_3}, divide by $h$ and let $h\rightarrow 0$. By Lebesgue's differentiation theorem we get
	\begin{equation*}
	\begin{split}
		\MoveEqLeft{\lim_{h\rightarrow 0^+} \frac{\mathcal{D}_c(\mu_1(t,\cdot),\mu_2(t,\cdot))-\mathcal{D}_c(\mu_1(t-h,\cdot),\mu_2(t-h,\cdot))}{h} }\\
		& \leq \intrd \nabla\zeta_t\cdot (u_1(t,x)\mu_1(t,x)-u_2(t,x)\mu_2(t,x))dx  - \intrd \nabla\zeta_t \cdot(\kappa_1\nabla\mu_1(t,x)-\kappa_2\nabla\mu_2(t,x)) dx
	\end{split}
	\end{equation*}
	and
	\begin{equation*}
	\begin{split}
		\MoveEqLeft{\lim_{h\rightarrow 0^-} \frac{\mathcal{D}_c(\mu_1(t,\cdot),\mu_2(t,\cdot))-\mathcal{D}_c(\mu_1(t-h,\cdot),\mu_2(t-h,\cdot))}{h}}\\
		& \geq \intrd \nabla\zeta_t\cdot (u_1(t,x)\mu_1(t,x)-u_2(t,x)\mu_2(t,x))dx - \intrd \nabla\zeta_t \cdot(\kappa_1\nabla\mu_1(t,x)-\kappa_2\nabla\mu_2(t,x)) dx
	\end{split}
	\end{equation*}
	which implies \eqref{eq:dtD} for almost every $t\in(0,T)$.
\end{proof}

\section{Stability in the DiPerna--Lions setting: Proof of Theorem~\ref{thm:1}}
\label{s:main}

Next in order, we will introduce a first tool that is essential in the proof of Theorem \ref{thm:1}. 
It is based on the Hardy--Littlewood maximal function operator $M$, which is a central tool from the Calderón--Zygmund theory defined for any measurable functions $f:\rd\rightarrow\mathbb{R}$ by
\begin{equation}
Mf(x) = \sup_{R>0} \frac{1}{|B_{R}(x)|}\int_{B_{R}(x)} |f(z)|dz.
\label{eq:maximal}
\end{equation}
The operator is continuous from $L^p$ to $L^p$ if $p\in (1,\infty]$, hence $\|Mf\|_{L^p} \lesssim \|f\|_{L^p}$, see \cite{Stein70} for details. Moreover with the maximal function operator one can find bounds for the difference quotients  by using Morrey-type inequalities, namely, for almost any $x,y\in\rd$ it holds
\begin{equation}
	\frac{|f(x)-f(y)|}{|x-y|} \lesssim M(\nabla f)(x) + M(\nabla f)(y).
	\label{eq:morrey}
\end{equation}
A proof of the continuity of the Morrey-type estimate can be found in \cite{EvansGariepy92}.
The estimate in Lemma~\ref{lem:maximal} below resembles the one in \cite[Lemma 3]{Seis17}, but here we adapt the setting and the notation to make it more useful for our purpose.
\begin{lemma}\label{lem:maximal}
Let $p\in(1,\infty]$ and $q\in[1,\infty)$ such that $1/p+1/q=1$. Let $\eta_1,\eta_2\in L^1\cap L^q$ be densities of equal mass, i.e. $\intrd \eta_1 = \intrd \eta_2$, and let $\sigma\in\Pi(\eta_1,\eta_2)$ be a coupling with marginals $\eta_1$ and $\eta_2$. Then for any integrable function $u$ with $\nabla u\in L^p$, it holds
\begin{equation}
\iint_{\rd\times\rd} \frac{|u(x)-u(y)|}{|x-y|}d\sigma(x,y) \lesssim (\|\eta_1\|_{L^q}+\|\eta_2\|_{L^q})\|\nabla u\|_{L^p}.
\end{equation}
\end{lemma}
\begin{proof}
First, the statement is trivial if $p=\infty$ since $u$ is Lipschitz and we arrive straightforwardly to the result of the lemma. 
Now, let $p\in(1,\infty)$. By means of the Morrey-type inequality \eqref{eq:morrey} and using the marginal conditions for the measure $\sigma(x,y)$ we obtain
\begin{equation*}
\begin{split}
\iint_{\rd\times\rd} \frac{|u(x)-u(y)|}{|x-y|}d\sigma(x,y) & \lesssim \iint_{\rd\times\rd} (M(\nabla u)(x) + M(\nabla u)(y))d\sigma(x,y) \\
& = \intrd M(\nabla u)(x)\eta_1(x)dx + \intrd M(\nabla u)(y)\eta_2(y)dy.
\end{split}
\end{equation*}
Then by Hölder inequality and by the continuity of the maximal operator from $L^p$ to $L^p$ we deduce the statement of the Lemma.
\end{proof}

Observe that this result can only be applied when $\nabla u\in L^p$ with $p>1$. The limit case where $\grad u$ is a singular integral of an $L^1$ function, that will be of interest in the next section, has to be dealt more carefully and requires some more elaborate tools from the Calderón--Zygmund theory.

\begin{proof}[Proof of Theorem \ref{thm:1}] We consider the dual setting of the Kantorovich--Rubinstein distance between the solutions $\theta_1$ and $\theta_2$. By Lemma \ref{lemma:dtD} we can write and set
\begin{align}
\MoveEqLeft\frac{d}{dt}\mathcal{D}_\delta(\theta_1,\theta_2)  = \intrd \nabla\zeta_t\cdot (u_1\theta_1-u_2\theta_2)dx - \intrd \nabla\zeta_t \cdot(\kappa_1\nabla\theta_1-\kappa_2\nabla\theta_2) dx \nonumber\\
& = \intrd \nabla\zeta_t\cdot (u_1\theta_1-u_2\theta_2)dx -  \kappa_1\intrd \grad \zeta_t \cdot (\grad\theta_1-\grad\theta_2) dx  + (\kappa_2-\kappa_1)\intrd \nabla\zeta_t \cdot \nabla\theta_2 dx\nonumber \\
& =: \Theta_1(t) + \Theta_2(t) + \Theta_3(t).\label{eq:sum_Thetas}
\end{align}
We prove that the individual terms are controlled as follows
\begin{align}
	\Theta_1(t) &\lesssim \|\grad u_1\|_{L^p} + \frac{\|u_1-u_2\|_{L^p}}{\delta} ;
	\label{eq:t1}\\
	\Theta_2(t) &\leq 0 ; \label{eq:t2} \\
	\Theta_3(t) &\leq \frac{|\kappa_2-\kappa_1|}{\delta}\|\nabla\theta_2\|_{L^1}.		\label{eq:t3}
\end{align}
Before turning to the  proofs of these bounds, we can straightforwardly conclude the proof of Theorem~\ref{thm:1}.
Indeed, inserting the bounds \eqref{eq:t1}, \eqref{eq:t2} and \eqref{eq:t3}  into \eqref{eq:sum_Thetas} imply after integrating over $(0,t)$ for any $0\leq t\leq T$ the stability estimate
	\begin{equation*}
		\mathcal{D}_\delta(\theta_1(t,\cdot),\theta_2(t,\cdot)) - \mathcal{D}_\delta(\theta_1(0,\cdot),\theta_2(0,\cdot)) \lesssim \|\grad u_1\|_{L^1(L^p)} + \frac{\|u_1-u_2\|_{L^1(L^p)}}{\delta} + \frac{|\kappa_1-\kappa_2|}{\delta}\|\nabla\theta_2\|_{L^1},
	\end{equation*}
which is what we aimed to prove.

\emph{Proof of Estimate~\eqref{eq:t1}.}
For the first term $\Theta_1(t)$ we will use the dual representation of the optimal transportation distance. Using the properties of the Kantorovich potential $\zeta_t$ and the marginal conditions of the optimal transport plan, it holds
\begin{equation*}
\begin{split}
\Theta_1(t) & =  \intrd \nabla\zeta_t\cdot (u_1\theta_1-u_2\theta_2)\, dx\\
&  = \iint_{\rd\times\rd} (\nabla\zeta_t(x)\cdot u_1(x)-\nabla\zeta_t(y)\cdot u_2(y))d\pi_t(x,y) \\
& = \iint_{\rd\times\rd} \frac{1}{|x-y|+\delta}\frac{x-y}{|x-y|}\cdot (u_1(x)-u_2(y))d\pi_t(x,y),
\end{split}
\end{equation*}
where $\pi_t$ is the optimal transport plan in $\Pi((\theta_1(t)-\theta_2(t))^+,(\theta_1(t)-\theta_2(t))^-)$. We now separate the gradient term from the error term,
\begin{equation*}
\begin{split}
|\Theta_1(t)| & \leq \iint_{\rd\times\rd} \frac{|u_1(x)-u_1(y)|}{|x-y|} d\pi_t(x,y) + \frac{1}{\delta}\iint_{\rd\times\rd} |u_1(y)-u_2(y)| d\pi_t(x,y) .
\end{split}
\end{equation*}
For the difference quotients in the first term, we apply Lemma \ref{lem:maximal}. Regarding the second term, we can use the marginal conditions and the Hölder inequality,
\begin{align*}
\frac{1}{\delta}\iint_{\rd\times\rd} |u_1(y)-u_2(y)| d\pi_t(x,y)& = \frac{1}{\delta}\intrd |u_1-u_2|(\theta_1- \theta_2)^-dy \leq \frac{\|u_1-u_2\|_{L^p}}{\delta}\|\theta_2\|_{L^q}.
\end{align*}
All in all, we have established \eqref{eq:t1}.

\emph{Proof of Estimate~\eqref{eq:t2}.}
The control of $\Theta_2(t)$, the second term in \eqref{eq:sum_Thetas}, is based on a discretization approach and hence we introduce \emph{finite difference quotients}. 
Assume $v:\rd\rightarrow \mathbb{R}$ is a locally summable function, then the $i^{th}$ difference quotient with $1\leq i\leq d$ of size $h>0$ at $x\in \R^d$ is given by
\begin{equation}
  D_i^h v(x) = \frac{v(x+he_i)-v(x)}{h}.
  \label{eq:diff_quotient}
\end{equation}
We also make use of these quotients to approximate the Laplacian by the standard three point stencil
\begin{equation*}
	  \Delta^h v(t,x) :=  \sum_{i=1}^d -D_i^{-h}D_i^v \theta(t,x) = \frac1{h^2} \sum_{i=1}^d \left(v(x + he_i) - 2v(x) + v(x-he_i)\right).\end{equation*}
We follow a technique inspired from \cite{FournierPerthame19} consisting of a convenient rearrangement of the terms involved thanks to a discretization of the spatial derivatives. Let us take $h>0$ sufficiently small and write the finite differences approximation of the Laplacian using the difference quotients to arrive at a discretized analog of $\Theta_2$ given for all $t\in(0,T)$ by
\begin{align*}
	\frac1{\kappa_1}\Theta_2^h(t) &=  \intrd \zeta_t (x)\; \Delta^h \theta_1(t,x)dx -  \intrd \zeta_t(x) \; \Delta^h \theta_2(t,x) dx \\
	&= \frac{1}{h^2}\intrd \zeta_t (x)\sum_{i=1}^d  \bigl(\theta_1(t,x+he_i)-2\theta_1(t,x)+\theta_1(t,x-he_i)\bigr)dx \\
		&\quad - \frac{1}{h^2}\intrd \zeta_t (x)\sum_{i=1}^d  \bigl(\theta_2(t,x+he_i)-2\theta_2(t,x)+\theta_2(t,x-he_i)\bigr)dx.
\end{align*}
Since $\theta_1,\theta_2 \in W^{1,1}$ and $\zeta \in  W^{1,\infty}$, we have that
\begin{equation*}
\Theta_2(t) = \lim_{h\rightarrow 0}\Theta_2^h(t) ,
\end{equation*}
and thus, it is enough to estimate $\Theta^h_2$ instead of $\Theta_2$. 

In order to estimate $\Theta_2^h$, we apply a convenient changes of variables $z_i=x+he_i$ and $y_i=x-he_i$ for all $1\leq i\leq d$ to arrive at
\begin{align*}
	\frac1{\kappa_1}\Theta_2^h(t)
	&= \frac{1}{h^2}\intrd \left(\theta_1(t,x) -\theta_2(t,x) \right)\sum_{i=1}^d  \bigl(\zeta_t(x+he_i)-2\zeta_t(x)+\zeta_t(x-he_i)\bigr)dx ,
\end{align*}
and exploring to the optimality of $\zeta_t $ 
therefore now $\zeta_{t,z_i}$ and $\zeta_{t,y_i}$ in the dual formulation of the Kantorovich--Rubinstein distance $\mathcal{D}_\delta (\theta_1,\theta_2)$, we obtain
\begin{equation*}
\frac{h^2}{\kappa_1}\Theta_2^h(t) \leq -2d\mathcal{D}_\delta (\theta_1, \theta_2) + d\mathcal{D}_\delta (\theta_1, \theta_2) + d\mathcal{D}_\delta (\theta_1, \theta_2) = 0,
\label{eq:t2h}
\end{equation*}
which proves \eqref{eq:t2}. 

\emph{Proof of Estimate~\eqref{eq:t3}.}
Finally, for $\Theta_3(t)$ we can use again that $\zeta_t$ is a Lipschitz function with Lipschitz constant bounded by $\|\nabla\zeta_t\|_{L^\infty}\leq 1/\delta$, so we arrive at~\eqref{eq:t3}.
\end{proof}

\section{Uniqueness of distributional solutions with rough advection fields: Proof of Theorem~\ref{thm:2}}
\label{ss:uniqueness}

In this section, we are going to deal with the existence and uniqueness problem for the advection-diffusion equation \eqref{eq:adv-diff} stated in Theorem \ref{thm:2}. Before turning to its proof, we briefly verify that the velocity fields considered here belong indeed to a  weak $L^p$ space, globally in space. 
\begin{lemma}\label{L100}
Let $u$ be given by~\eqref{eq:Hu2} with $\omega \in L^1(\R^d)$ and $k$ satisfying~\ref{ass:k1}--\ref{ass:k4}, then it holds that
\[
u\in L^{p,\infty}\bigr(\R^d\bigl)\qquad\text{ for }   p = \frac{d}{d-1}.
\]
\end{lemma}
Our assumption  in~\eqref{eq:Hu1} is a little bit stronger than what is implied by \eqref{103}, but it falls into the class of velocity fields that are, for instance, induced by $L^{\infty}(L^1)$ vorticity solutions to the Euler equations as considered in \cite{CrippaNobiliSeisSpirito17}.

The result is a classical result from harmonic analysis. We provide a short elementary proof for the convenience of the reader.
\begin{proof}By rescaling $\omega$, we may without loss of generality assume that $\|\omega\|_{L^1}=1$. In view of the assumptions on the velocity field, we have the pointwise estimate
\[
|u(x)| \lesssim \int_{\R^d}\frac{|\omega(y)|}{|x-y|^{d-1}}dy.
\]
Given some radius $R>0$ that we will fix later, we now decompose the integral into a bounded and an integrable part,
\begin{equation}\label{105}
|u(x)| \lesssim  \int_{B_R(x)}\frac{|\omega(y)|}{|x-y|^{d-1}}dy +  \int_{B_R(x)^c}\frac{|\omega(y)|}{|x-y|^{d-1}}dy \le \left(\chi_{B_R(0)} \frac1{|\cdot|^{d-1}}\right)\ast |\omega| (x) + \frac{1}{R^{d-1}}.
	\end{equation}
An integral bound on the first term is obtained via an elementary computation that is based on Young's inequality,
\begin{equation}\label{106}
\left\lVert\left(\chi_{B_R(0)} \frac1{|\cdot|^{d-1}}\right)\ast \omega\right\rVert_{L^1} \le \biggl\lVert\chi_{B_R(0)} \frac1{|\cdot|^{d-1}}\biggr\rVert_{L^1}\|\omega \|_{L^1}\lesssim R.
\end{equation}
Here, we do not keep track of the dependence of constants on $\omega$.

In order to estimate the weak $L^p$ norm, we let $\lambda $ be arbitrarily given and suppose that $\lambda <|u(x)|$. Then we have by \eqref{105} for some $R$ such that $\lambda \sim R^{{1-d}}$ , 
\[
c \lambda \le  \left(\chi_{B_R(0)} \frac1{|\cdot|^{d-1}}\right)\ast |\omega|(x),
\]
for some  small $c$, and thus, 
\[
\L^{d+1}\bigl(\{|u|>\lambda\}\bigr)\le \L^{d+1}\biggl(\biggl\{\chi_{B_R(0)}\frac1{|\cdot|^{d-1}} \ast | \omega |>c\lambda\biggr\}\biggr) \lesssim\frac1{\lambda}\left\lVert\left(\chi_{B_R(0)} \frac1{|\cdot|^{d-1}}\right)\ast |\omega|\right\rVert_{L^1} . 
\]
As a consequence of \eqref{106} and the above choice of $R$, we deduce that
\[
\lambda^p \L^{d+1}\bigl(\{|u|>\lambda\}\bigr)\lesssim \lambda^{p-1-\frac1{d-1}}.
\]
Choosing $p$ as in the statement of the lemma, we see that the right-hand side is in fact independent of $\lambda$ and so is the supremum in $\lambda$, which yields the desired bound.
\end{proof}
Towards a proof of the uniqueness result from Theorem \ref{thm:2}, we will have to establish a suitably adapted version of the stability estimate in Theorem \ref{thm:1}. 
\begin{proposition}\label{prop:stab_est_2}
Under the assumptions of Theorem \ref{thm:exist_uniq}, let $\theta\in L^\infty((0,T);(L^1\cap L^\infty)(\rd))$ be a solution of \eqref{eq:adv-diff} with $\intrd \theta^0 = 0$ and $\theta\not\equiv 0$. Then for every $\varepsilon>0$ there exists a constant $C_\varepsilon>0$ such that for every $\delta >0$ it holds
\begin{equation*}
\sup_{0\leq t\leq T}\mathcal{D}_\delta(\theta(t,\cdot)) \lesssim \mathcal{D}_\delta(\theta^0) + \varepsilon\|\theta\|_{L^1}\left[1+\log\left( \frac{1}{\varepsilon\delta}\left(\frac{\|\theta\|_{L^1}}{\|\theta\|_{L^\infty}}\right)^{1-\frac{1}{p}}\|u\|_{L^{p,\infty}} \right)\right] + C_\varepsilon\|\theta\|_{L^\infty(L^2)}.
\end{equation*}
\end{proposition}
This estimate was derived earlier in the non-diffusive setting \cite{CrippaNobiliSeisSpirito17}. As the argument in the present work is essentially a combination of the one therein and the one that we proposed in order to establish Theorem \ref{thm:1}, we will keep our presentation here short.

Regarding the proof,  the main difference between the DiPerna--Lions setting considered in the previous section and the one we deal with here is the failure of the maximal function estimates. Instead of estimating difference quotients with the help of the Morrey-type estimate in \eqref{eq:morrey}, the strategy here is to construct certain weighted maximal functions which allow for the substitutive estimate
\begin{equation}
\frac{|u(t,x)-u(t,y)|}{|x-y|} \lesssim G(t,x)+G(t,y) \qquad \text{ for all } x,y\not\in N_t,
\label{eq:weak_morrey2}
\end{equation}
where $N_t$ is a negligible set,  $\mathcal{L}^d(N_t)=0$, which exists for almost every time $t$, and $G:(0,T)\times \rd\rightarrow \mathbb{R}$ is a function which can be decomposed for every $\varepsilon>0$ into a sum $G=G_\varepsilon^1+G_\varepsilon^2$, where $G_{\eps}^1$ and $G_{\eps}^2$ are such that
\begin{equation}
\|G_\varepsilon^1\|_{L^1(L^{1,\infty})} \leq \varepsilon, \qquad \|G_\varepsilon^2\|_{L^1(L^2)} \leq C_\varepsilon.
\label{eq:weak_morrey1}
\end{equation}
If $\omega\in L^1(L^1)$, it is proved in \cite{BouchutCrippa13} that such a function exists and the constant $C_\varepsilon$ depends not only on $\varepsilon>0$ but also on the equi-integrability of $\omega$. Therefore, this result would not generalize to situations in which $\omega$ is simply a measure.

We will not give any details about the construction of the function $G$ in \eqref{eq:weak_morrey2} and \eqref{eq:weak_morrey1}. However, for the convenience of the reader, we provide here the full argument for Proposition \ref{prop:stab_est_2}.
\begin{proof}[Proof of Propositon \ref{prop:stab_est_2}]
As we are seeking the stability estimate analogous to the one in Theorem~\ref{thm:1}, we start by computing the time derivative of the optimal transportation distance. Denoting $\zeta_t$ the Kantorovich potential at time $t$, by Lemma \ref{lemma:dtD} we have
\begin{equation*}
\frac{d}{dt}\mathcal{D}_\delta(\theta(t,\cdot)) = \intrd \nabla\zeta_t\cdot u(t,x)\theta(t,x)dx - \kappa\intrd \nabla\zeta_t \cdot\nabla\theta(t,x)dx.
\end{equation*}
Hence, by a similar analysis to the performed in Theorem \ref{thm:1}, we obtain the following estimate
\begin{equation}
\sup_{t\in (0,T)} \mathcal{D}_\delta(\theta(t,\cdot)) \leq \mathcal{D}_\delta(\theta^0) + \int_0^T\iint_{\rd\times\rd} \frac{|u(t,x)-u(t,y)|}{|x-y|+\delta} d\pi_t(x,t).
\label{eq:first_estimate_p1}
\end{equation}
In this case we cannot apply the Morrey inequality \eqref{eq:morrey} as before. Nonetheless, we might make use of the alternative theory developed for weak $L^p$ spaces and apply the inequalities \eqref{eq:weak_morrey2} and \eqref{eq:weak_morrey1}. Therefore we can estimate pointwise the integrand in \eqref{eq:first_estimate_p1} for almost every $t\in (0,T)$ and every $x,y\in N_t$, where $\mathcal{L}^d(N_t)=0$ by
\begin{equation}
\frac{|u(t,x)-u(t,y)|}{|x-y|+\delta} \lesssim \min\left\lbrace \frac{|u(t,x)|+|u(t,y)|}{\delta}, G_\varepsilon^1(t,x)+G_\varepsilon^1(t,y) \right\rbrace + G_\varepsilon^2(t,x)+G_\varepsilon^2(t,y) .
\label{eq:uniq_step1}
\end{equation}
Notice that, since the marginals of the optimal transport plan $\pi_t$ are absolutely continuous with respect to $\mathcal{L}^d$, the pointwise estimate holds for almost every $t\in(0,T)$ and $\pi_t$-almost every $(x,y)\in\rd\times\rd$. 

To begin with, we take care of the first term in the right-hand side. We introduce this first part of the estimate into \eqref{eq:first_estimate_p1} and split the terms as follows,
\begin{equation*}
\int_0^T\iint_{\rd\times\rd} \min\left\lbrace \frac{|u(t,x)|+|u(t,y)|}{\delta}, G_\varepsilon^1(t,x)+G_\varepsilon^1(t,y) \right\rbrace d\pi_t(x,y)dt \leq I_1 + I_2,
\end{equation*}
where
\begin{align*}
I_1 & = \int_0^T\iint_{\rd\times\rd} \left( \min\left\lbrace \frac{|u(t,x)|}{\delta}, G_\varepsilon^1(t,x) \right\rbrace + \min\left\lbrace \frac{|u(t,y)|}{\delta}, G_\varepsilon^1(t,y) \right\rbrace \right)d\pi_t(x,y)dt , \\
I_2 & = \int_0^T\iint_{\rd\times\rd} \left( \min\left\lbrace \frac{|u(t,x)|}{\delta}, G_\varepsilon^1(t,y) \right\rbrace + \min\left\lbrace \frac{|u(t,y)|}{\delta}, G_\varepsilon^1(t,x) \right\rbrace\right) d\pi_t(x,y)dt .
\end{align*}
By means of the marginal condition for the optimal transport plan we have
\begin{equation*}
I_1 = \int_0^T\intrd \min\left\lbrace \frac{|u(t,x)|}{\delta},G_\varepsilon^1(t,x) \right\rbrace |\theta(t,x)| dxdt.
\end{equation*}
In order to make the notation simpler, we write $\psi = \min\{|u|/\delta,G_\varepsilon^1\}$. The main challenge now comes from the fact that $G_\varepsilon^1$ can only be bounded in the weak space $L^{1,\infty}$ while $u$ in $L^{p,\infty}$. Let us define the finite measure
\begin{equation*}
d\mu(t,x) = \chi_{(0,T)}(t)|\theta(t,x)|\,d\bigl(\mathcal{L}^1\otimes\mathcal{L}^d\bigr)
\end{equation*}
in $\mathbb{R}^{d+1}$ so that, since $\psi$ is defined as a minimum, we can bound on the one hand
\begin{equation}
\|\psi\|_{L^{1,\infty}(\mu)} \leq \|G_\varepsilon^1\|_{L^{1,\infty}(\mu)} \leq \|\theta\|_{L^\infty}\|G_\varepsilon^1\|_{L^1(L^{1,\infty})} \leq \varepsilon \|\theta\|_{L^\infty}
\label{eq:uniq_step2}
\end{equation}
and on the other hand by using \eqref{eq:weak_morrey1} also
\begin{equation}\label{eq:uniq_step3}
\|\psi\|_{L^{p,\infty}(\mu)} \leq \frac{1}{\delta}\|u\|_{L^{p,\infty}(\mu)} \leq \frac{1}{\delta}\|\theta\|_{L^{\infty}}^{1/p}\|u\|_{L^{p,\infty}}.
\end{equation}
Now, using from Lemma 2.6 in \cite{CrippaNobiliSeisSpirito17} the interpolation inequality
	\begin{equation*}
		\|\psi\|_{L^1(\mu)} \leq \frac{p}{p-1}\|\psi\|_{L^{1,\infty}(\mu)} \left[ 1+\log\left(\frac{\mu(\R^{d+1})^{1-\frac{1}{p}}\|\psi\|_{L^{p,\infty}(\mu)}}{\|\psi\|_{L^{1,\infty}(\mu)}}\right)  \right],
	\end{equation*}
 and the  monotonicity of the expression on the right-hand side in the $L^{1,\infty}$-norm, we find for any $\theta\not\equiv 0$ the bound
\begin{equation}
I_1 = \|\psi\|_{L^1(\mu)} \lesssim \varepsilon\|\theta\|_{L^\infty} \left[1+\log\left( \frac{1}{\varepsilon\delta}\left(\frac{\|\theta\|_{L^1}}{\|\theta\|_{L^\infty}}\right)^{1-\frac{1}{p}}\|u\|_{L^{p,\infty}} \right)\right] .
\label{eq:i1_estimate}
\end{equation}

The argument for $I_2$ will be similar to the just performed estimates. To do so, we need to prove that the estimates \eqref{eq:uniq_step2} and \eqref{eq:uniq_step3} hold also in the situation of $I_2$. Recall the characterization of the optimal transport plans through the measurable maps $T$ and $S$, introduced in Section \ref{s:opt}, which together with the marginal condition give
\begin{equation*}
I_2 = \int_0^T\!\!\intrd \!\left( \min\left\lbrace \frac{|u\circ S|(t,y)}{\delta}, G_\varepsilon^1(t,y)\right\rbrace\theta^-(t,y) +\min\left\lbrace \frac{|u\circ T|(t,x)}{\delta}, G_\varepsilon^1(t,x)\right\rbrace\theta^+(t,x) \right) dydt .
\end{equation*}
The treatment of the two terms now is quite similar, therefore it would be enough to focus on one of them, say the last one. Analogously to the estimate for $I_1$, we define now the function
\begin{equation*}
\psi = \min\left\lbrace \frac{|u\circ T|}{\delta}, G_\varepsilon^1 \right\rbrace,
\end{equation*}
and the finite measure $d\mu(t,x) = \chi_{(0,T)}(t)\theta^+(t,x)d\mathcal{L}^1\otimes\mathcal{L}^d$ on $\mathbb{R}^{d+1}$. The first estimate \eqref{eq:uniq_step2} remains valid without any change since it comes from assuming $\psi\leq G_\varepsilon^1$, that holds true. For the second estimate \eqref{eq:uniq_step3}, however, we need to take care of some details. On the one hand, we have
\begin{equation*}
\|\psi\|_{L^{p,\infty}(\mu)} \leq \frac{1}{\delta}\|u\circ T\|_{L^{p,\infty}(\mu)},
\end{equation*}
and on the other hand, we have the relation $\theta^- = T_\# \theta^+$, which implies
\begin{equation*}
\mu(\{|u\circ T| > \lambda\}) = \left( T_\# \theta^+ \mathcal{L}^1\otimes\mathcal{L}^d \right)(\{|u| > \lambda\}) = \left( \theta^- \mathcal{L}^1\otimes\mathcal{L}^d \right)(\{|u| > \lambda\}).
\end{equation*}
Therefore, we have
\begin{equation*}
\begin{split}
\|u\circ T\|_{L^{p,\infty}(\mu)} & = \sup_{\lambda >0} \left(\lambda^p \int_0^T\intrd \chi_{\{|u\circ T| > \lambda\}}\theta^+(t,x)dxdt\right)^{1/p} \\
& = \sup_{\lambda >0} \left(\lambda^p \int_0^T\intrd \chi_{\{|u| > \lambda\}}\theta^-(t,x)dxdt\right)^{1/p} \leq \|\theta\|_{L^\infty}^{1/p}\|u\|_{L^{p,\infty}}.
\end{split}
\end{equation*}
Hence the estimate \eqref{eq:uniq_step3} also holds for $I_2$ and we arrive to the estimate \eqref{eq:i1_estimate} in this case as well.

Finally, we can control the terms related to $G_\varepsilon^2$ in \eqref{eq:uniq_step1} by means of the marginal conditions for the optimal transport plan and Hölder inequality,
\begin{equation}
\int_0^T\intrd (G_\varepsilon^2(t,x)+G_\varepsilon^2(t,y))d\pi_t(x,y) = \int_0^T\intrd G_\varepsilon^2|\theta|dtdx \leq \|\theta\|_{L^\infty(L^2)}\|G_\varepsilon^2\|_{L^1(L^2)}.
\label{eq:g2}
\end{equation}
that is bounded since $\theta\in L^\infty(L^2)$ by interpolation. Therefore we can plug the estimates \eqref{eq:i1_estimate} and \eqref{eq:g2} into \eqref{eq:uniq_step1} and it yields the desired stability estimate.
\end{proof}
From Proposition \ref{prop:stab_est_2} it is easy to deduce uniqueness with the help of Lemma \ref{lemma:bound_otd}. It remains to prepare for the proof of existence. 
We will establish existence by a standard mollification-and-compactness procedure. 
For this, we provide an auxiliary lemma about the convergence of the velocity fields in  appropriate Lebesgue spaces that we present here in a quantitative way. We believe this result is  of independent mathematical interest.
\begin{lemma}\label{lemma:u_convergence}
Consider $\omega_n,\omega:\rd\rightarrow\rd$ integrable functions all with the same total mass. Assume that $\|\omega_n - \omega\|_{L^1(\rd)}$ is bounded uniformly in $n\in\mathbb{N}$. Let $u_n = k\ast \omega_n$, $u=k\ast \omega$ with $k$ satisfying~\ref{ass:k1}--\ref{ass:k4}, then for every $s>0$ and any $1\leq p < d/(d-1)$ it holds 
\begin{equation*}
\|u_n-u\|_{L^p(B_s(0))} \lesssim \bigl(s^{\frac{d}{p}}W_1(\omega_n,\omega)\bigr)^{\frac{d-p(d-1)}{d+p}}.   
\end{equation*}
\end{lemma}
This lemma states that under a suitable convergence assumptions for $\omega_n$ in $L^1(\rd)$ we can control the $L^p_{loc}(\rd)$ convergence of $u_n$ in terms of  the Wasserstein distance for $1\leq p < d/(d-1)$.
The most evident consequence of this lemma is the following result, whose proof is now obvious.
\begin{lemma}
Under the same assumptions of Lemma \ref{lemma:u_convergence}, if $W_1(\omega_n,\omega)\rightarrow 0$ as $n\rightarrow\infty$, then $u_n$ converges to $u$ locally in  $L^p(\rd)$, provided that $1\leq p < d/(d-1)$.
\label{cor:u_convergence}
\end{lemma}
\begin{proof}[Proof of Lemma \ref{lemma:u_convergence}]
First, we show that $\omega\in L^1(\rd)$ implies $u\in L^p_{loc}(\rd)$. For this purpose, it is convenient and enough to assume that $\omega$ is nonnegative density of mass one, $\|\omega\|_{L^1}=1$. We choose $s>0$ arbitrary, and derive  by means of the Jensen inequality applied with the measure $\omega(y)dy$ and Fubini's theorem,
\begin{equation*}
\begin{split}
\int_{B_s(0)} |u|^pdx & \lesssim  \int_{B_s(0)}\left(\intrd \frac{\omega(y)}{|x-y|^{d-1}}dy\right)^p dx\lesssim \int_{B_s(0)}\intrd \frac{\omega(y)}{|x-y|^{p(d-1)}}dydx \\
& = \intrd\left(\int_{B_s(0)} \frac{1}{|x-y|^{p(d-1)}}dx\right)\omega(y)dy \lesssim s^{d - p(d-1)}
\end{split}
\end{equation*}
provided that $1\leq p < d/(d-1)$. In order to prove the strong convergence of $(u_n)_{n\in\mathbb{N}}$ to $u$ in $L^p_{loc}(\rd)$ we first study the pointwise difference $u_n(x)-u(x)$. Let $R>0$ be such that $k(z)$ is Lipschitz for all $z\in B_R(0)^c$, and define the cutoff functions $\eta_R:\rd\rightarrow [0,1]$, $\eta_R\in C_c^\infty(\rd)$, $\eta_R(x) = 1$ for all $x\in B_R(0)$ and $\eta_R(x)=0$ for all $x\in\overline{B_{2R}(0)}^c$. Then $\varphi(x)=(1-\eta_R)k(x)$ is a Lipschitz function with Lipschitz constant bounded by $\|\nabla\varphi\|_{L^\infty(\rd)} \lesssim R^{-d}$. Thus, we can write
\begin{equation*}
\begin{split}
u_n(x)-u(x) & = \intrd k(x-y)[\omega_n(y)-\omega(y)]dy \\
& = \intrd (\eta_Rk)(x-y)[\omega_n(y)-\omega(y)]dy + \intrd \varphi(x-y)[\omega_n(y)-\omega(y)]dy.
\end{split}
\end{equation*}
Since $\varphi(x)$ is Lipschitz, the second term in the right hand side can be related to the $1$-Wasserstein distance~\eqref{eq:wasserstein}, provided $\omega_n$ and $\omega$ are of the same total mass,
\begin{equation*}
\left|\intrd \varphi(x-y)[\omega_n(y)-\omega(y)]dy \right|\lesssim  \frac{1}{R^d}W_1(\omega_n,\omega).
\end{equation*}
On the other hand, the remaining term can be bounded by
\begin{equation*}
\intrd (\eta_Rk)(x-y)[\omega_n(y)-\omega(y)]dy \lesssim \int_{B_{2R}(x)} \frac{1}{|x-y|^{d-1}}|\omega_n(y)-\omega(y)|dy = (F_{2R}\ast |\omega_n-\omega|)(x),
\end{equation*}
where the function $F_{r}$ is defined as $F_{r}(z) = |z|^{1-d}\chi_{B_{r}(0)}(z)$ for any $r>0$. Thus, for every $s>0$ it holds
\begin{equation*}
\|u_n-u\|_{L^p(B_s(0))} \lesssim \|F_{2R}\ast|\omega_n-\omega| \|_{L^p(\R^d)}+\|R^{-d}W_1(\omega_n,\omega)\|_{L^p(B_s(0))}.
\end{equation*}
For the first term, we might use the Young's convolution inequality so that we get
\begin{equation*}
\|F_{2R}\ast|\omega_n-\omega| \|_{L^p(\R^d)} \leq \|F_{2R}\|_{L^p(\R^d)}\|\omega_n-\omega\|_{L^1(\rd)} \lesssim R^{\frac{d-p(d-1)}p}
\end{equation*}
given that $\|\omega_n-\omega\|_{L^1(\rd)}$ is bounded uniformly in $n\in\mathbb{N}$ and that $1\leq p< d/(d-1)$. The second term on the right-hand side is the norm of a constant, thus
\begin{equation*}
	\|u_n-u\|_{L^p(B_s(0))} \lesssim R^{\frac{d-p(d-1)}{p}} + s^{\frac{d}{p}}R^{-d}W_1(\omega_n,\omega).
\end{equation*}
But we can optimize the bound in $R>0$ to the effect that for all $s>0$ it holds
\begin{equation*}
\|u_n-u\|_{L^p(B_s(0))} \lesssim \Bigl(s^{\frac{d}{p}}W_1(\omega_n,\omega)\Bigr)^{\frac{d-p(d-1)}{d+p}}. \qedhere 
\end{equation*}
\end{proof}
We are now in the position to prove Theorem \ref{thm:2}.
\begin{proof}[Proof of Theorem \ref{thm:exist_uniq}]
First of all, we want to give a sketch of an  existence proof.  On this regard, we notice that distributional solutions are well-defined because $\nabla\theta\in L^1((0,T)\times\rd)$, see Remark \ref{R1}, and because $u\theta\in L^1((0,T)\times \rd)$. The letter  follows from the estimate
\begin{equation}
\|u\theta\|_{L^1} = \|u\|_{L^1(\mu)} \leq \frac{p}{p-1}\|\theta\|_{L^1}^{1-\frac{1}{p}}\|u\|_{L^{p,\infty}(\mu)} \leq \frac{p}{p-1}\|\theta\|_{L^1}^{1-\frac{1}{p}}\|\theta\|_{L^\infty}^{\frac{1}{p}}\|u\|_{L^{p,\infty}}<+\infty,
\label{eq:uthL1}
\end{equation}
where the measure $\mu$ is defined by $\mu(t,x) = \chi_{(0,T)}(t)\theta(x)\mathcal{L}^1\otimes \mathcal{L}^d$ --- we assume that $\theta$ is nonnegative for convenience --- and the first inequality is due to the embedding $ L^{p,\infty}(X,\mu)\subset L^1(X,\mu)$ on a finite measure space $(X,\mu)$, see, e.g., Lemma 2.5 in \cite{CrippaNobiliSeisSpirito17}.

Now, to prove existence, we will proceed by regularizing the velocity field and initial datum and then passing to the limit under the appropriate conditions. Denoting by $\rho_{\eps}$ a standard mollifier on $\R^d$, 
we define $\omega_\varepsilon = \omega\ast\rho_\varepsilon\in L^1((0,T);C^1_b(\rd))$ and $\theta^0_{\varepsilon} = \theta^0\ast\rho_\varepsilon\in C^1_b(\rd)$. Then, since $u=k\ast\omega$, we can also define $u_\varepsilon = k\ast \omega_\varepsilon \in L^1((0,T);C^1_b(\rd))$. Therefore, by standard theory, we know that there exist a unique solution $\theta_\varepsilon\in C((0,T);C^1_b(\rd))$ of the Cauchy problem
\begin{equation*}
\left\lbrace
\begin{array}{ll}
\partial_t\theta_\varepsilon + \nabla\cdot(u_\varepsilon\theta_\varepsilon) = \kappa\Delta \theta_\varepsilon & \text{ in } (0,T)\times\rd, \\
\theta_\varepsilon(0,\cdot) = \theta_{0,\varepsilon} & \text{ in } \rd.
\end{array}
\right.
\end{equation*}
Now we can use the elementary a priori estimates \eqref{eq:apriori_1} and our assumptions on the initial datum~\eqref{eq:initial} in order to deduce that $\theta_\varepsilon$ is bounded in $L^\infty((0,T);L^q(\rd))$ independently of $\varepsilon>0$ for every $1\leq q\leq \infty$. It follows that we can extract a subsequence (not relabelled) that converges weakly-$*$ to some function $\theta$ in $L^{\infty}(L^q)$ for any $q\in (1,\infty]$, and then by invoking some soft arguments, also in $L^{\infty}(L^1)$. Moreover, inspection of \eqref{eq:adv-diff} reveals that the time derivatives $\partial_t \theta_{\eps}$ are bounded  in $L^{\infty}(H^{-s})$ for some $s>0$, from which we infer that that the convergence takes place in $C^0($w-$L^q)$ for any $q\in [1,\infty]$, where $\mbox{w-}L^q(\R^d)$ is the standard $L^q$ space equipped with the weak topology. In view of Lemma \ref{cor:u_convergence} and \eqref{100} and because Wasserstein distances metrize weak convergence, see Theorem~7.12 in \cite{Villani03}, the velocity fields $u_{\eps}$ are converging locally in any $L^q$ space. As a consequence, the product $u_{\eps}\theta_{\eps}$ is convergent on compact sets, and thus, passing to the limit in the distributional formulation of \eqref{eq:adv-diff}, see Definition \ref{def:distr_sol}, we find that $\theta$ solves the advection-diffusion equation with velocity $u$ and initial datum $\theta^0$.

The proof of the uniqueness relies on the stability estimate from Proposition \ref{prop:stab_est_2}. Towards a contradiction, we assume that there is a solution $\theta(t,x)$ of the advection-diffusion equation \eqref{eq:adv-diff} with initial datum $\theta^0\equiv 0$ and such that $\theta(t,x) \not\equiv 0$, so that, in particular, $\|\theta\|_{L^\infty}>0$. Then we can write
\begin{equation*}
\sup_{0\leq t\leq T} \mathcal{D}_\delta (\theta(t,\cdot)) \lesssim \varepsilon\left[1+\log\left(\frac{1}{\varepsilon\delta}\right)\right] + C_\varepsilon,
\end{equation*}
where the symbol $\lesssim$ now includes $\|\theta\|_{L^\infty}$, $\|\theta\|_{L^1}$ and $\|u\|_{L^{p,\infty}}$. Notice that if $\delta \in (0,1/e)$ it holds
\begin{equation*}
\frac{1}{|\log\delta|}\left[1+\log\left(\frac{1}{\delta\varepsilon}\right)\right] \leq \frac{1+|\log\delta|+|\log\varepsilon|}{|\log\delta|} \leq 2+|\log\varepsilon|,
\end{equation*}
and, therefore, we can choose $a>0$ arbitrarily small and fix $\varepsilon>0$ such that
\begin{equation*}
\frac{\varepsilon}{|\log\delta|}\left[1+\log\left(\frac{1}{\delta\varepsilon}\right)\right] \leq \frac{a}{2}.
\end{equation*}
Since $\varepsilon>0$ and $C_\varepsilon>0$ are fixed now, we may choose $\delta\in(0,1/e)$ small enough so that
\begin{equation*}
\frac{C_\varepsilon}{|\log\delta|} \leq \frac{a}{2}.
\end{equation*}
Combining the previous estimates, we find that
\begin{equation*}
\sup_{0\leq t\leq T} \mathcal{D}_\delta (\theta(t,\cdot)) \lesssim a|\log\delta|.
\end{equation*}
Thus, since $a>0$ was arbitrarily small, it holds
\begin{equation*}
\frac{\mathcal{D}_\delta(\theta(t,\cdot))}{|\log\delta|} \rightarrow 0 \quad \text{as } \delta\rightarrow 0.
\end{equation*}
To conclude, it only remains to notice that Lemma \ref{lemma:bound_otd} with $\gamma = \sqrt{\delta}$ implies that
\begin{equation*}
\mathcal{D}^b(\theta(t,\cdot)) \leq 2\frac{\mathcal{D}_\delta(\theta(t,\cdot))}{|\log\delta|} + \sqrt{\delta}\|\theta(t,\cdot)\|_{L^1(\rd)}
\end{equation*}
for all $\delta>0$ small enough. In particular letting $\delta \rightarrow 0$ we get $\mathcal{D}^b(\theta(t,\cdot))=0$ and since $\mathcal{D}^b(\cdot)$ is a norm, it implies $\theta\equiv 0$. This contradicts the hypothesis at the beginning of the proof and since we found that the only solution of \eqref{eq:adv-diff} with initial datum $\theta^0=0$ is $\theta(t,x)=0$ almost every $(t,x)\in (0,T)\times\rd$, it yields the sought uniqueness.
\end{proof}

\section*{Acknowledgement} 
This work is funded by the Deutsche Forschungsgemeinschaft (DFG, German Research Foundation) under Germany's Excellence Strategy EXC 2044 --390685587, Mathematics M\"unster: Dynamics--Geometry--Structure and by the DFG Grant 432402380.

\bibliography{diffusive-mixing}
\bibliographystyle{acm}
\end{document}